\documentclass[a4paper]{article}
\usepackage{url}
\usepackage{amssymb,amsthm,amsmath,epsfig,cite}
\usepackage{graphicx}
\usepackage{float}
\usepackage{color}
\usepackage{amsmath}
\usepackage{amssymb}
\usepackage{amsfonts}
\usepackage{multirow}
\usepackage{subfigure}
\usepackage{bm}
\usepackage{caption}
\usepackage{algorithm} %format of the algorithm
\usepackage{algorithmic} %format of the algorithm
\usepackage{makecell}

\newcommand\blfootnote[1]{%
  \begingroup
  \renewcommand\thefootnote{}\footnote{#1}%
  \addtocounter{footnote}{-1}%
  \endgroup
}

\numberwithin{equation}{section}

\newtheorem{Theorem}{Theorem}[section]
\newtheorem{Lemma}{Lemma}[section]

\newtheorem{Example}{Example}[section]

\DeclareMathOperator{\rank}{rank} \DeclareMathOperator{\diag}{diag}

\DeclareMathOperator{\subspan}{span}

\title{A Modified Schur Method for Robust Pole Assignment\\ in State Feedback Control
\thanks{This research was supported in part by NSFC under grant 61075119 and
                     the Fundamental Research Funds for the Central Universities (BUPT2013RC0903).}}

\author{Zhen-chen Guo\thanks{LMAM \& School of Mathematical Sciences, Peking University, Beijing, 100871, China}
\quad Yun-feng Cai\footnotemark[2]
\quad  Jiang Qian\thanks{School of Sciences, Beijing University of Posts and Telecommunications, Beijing, 100876, China}
\quad Shu-fang Xu\footnotemark[2]
}

\begin{document}

\date{\today}

\maketitle

\blfootnote{Email addresses: guozhch06@gmail.com (Z.C. Guo),  yfcai@math.pku.edu.cn (Y.F. Cai),
 jqian104@gmail.com (J. Qian),  xsf@pku.edu.cn (S.F. Xu)}
%\tableofcontents

\begin{abstract}
Recently, a \textbf{SCHUR} method was proposed in \cite{Chu2} to solve the robust pole assignment problem in state feedback control.
It takes  the departure from normality of the closed-loop system matrix $A_c$ as the measure of robustness,
and intends to minimize it via the real Schur form of $A_c$.
The \textbf{SCHUR} method works well for real poles,
but when complex conjugate poles are involved, it does not produce the real Schur form of $A_c$ and can be problematic.
In this paper, we put forward a modified Schur method,
which improves the efficiency of  \textbf{SCHUR} when complex conjugate poles are to be assigned.
Besides producing the real Schur form of $A_c$, our approach also leads to a relatively small departure from
normality of $A_c$. Numerical examples show that our modified  method
produces better or at least comparable results than both \textbf{place} and \textbf{robpole} algorithms,
with much less computational costs.


\vskip2mm

\end{abstract}

\section{Introduction}
Let the matrix pair $(A, B)$ denotes the dynamic state equation
\begin{align}\label{eqstate-equation}
\dot{x}(t)=Ax(t)+Bu(t)
\end{align}
of the time invariant linear system, where $A\in \mathbb{R}^{n\times n}$ and $B\in \mathbb{R}^{n\times m}$
are the open-loop system matrix and the input matrix, respectively.
The dynamic behavior of \eqref{eqstate-equation} is governed by the eigen-structure of $A$, especially the poles (eigenvalues).
And in order to change the dynamic behavior of the open-loop system \eqref{eqstate-equation}
in some desirable way (to achieve stability or to speed up response),
one needs to modify the poles of  \eqref{eqstate-equation}.
Typically, this may be actualized by the state-feedback control
\begin{align}\label{eqfeedback}
u(t)=Fx(t),
\end{align}
where the feedback matrix $F\in \mathbb{R}^{m\times n}$ is to be chosen such that the closed-loop system
\begin{align}\label{eqstate-equation2}
\dot{x}(t)=(A+BF)x(t)\equiv A_c x(t)
\end{align}
has specified poles.

Mathematically,  the {\em  state-feedback pole assignment problem} can be stated as:

\noindent{\bf State-Feedback Pole Assignment Problem (SFPA)}
Given $A\in\mathbb{R}^{n\times n}$, $B\in\mathbb{R}^{n\times m}$ and
a set of $n$ complex numbers $\mathfrak{L}=\{\lambda_1, \lambda_2, \ldots, \lambda_n\}$,
closed under complex conjugation,
find an $F\in\mathbb{R}^{m\times n}$ such that $\lambda(A+BF)=\mathfrak{L}$,
where $\lambda(A+BF)$ is the eigenvalue set of $A+BF$.

A necessary and sufficient condition for the solvability of the {\bf SFPA} for any set $\mathfrak{L}$ of $n$ self-conjugate complex numbers is that $(A, B)$ is controllable, or equivalently,
the controllability matrix $\begin{bmatrix} B&AB&\cdots& A^{n-1}B\end{bmatrix}$ is of full row rank \cite{ XU, Won, Wonh}. Many
algorithms have been put forward to solve the {\bf SFPA}, such as
the invariant subspace method \cite{ PCK1},
the QR-like method \cite{ MP2, MP3}, etc..
We refer readers to \cite{ASC, BD, Chu0, Ka,RM, PM, Var, FO} for some other approaches.

When $m>1$, the solution to the {\bf SFPA} is generally not unique.
We may then utilize the freedom of $F$ to achieve some other desirable properties of the closed-loop system.
In applications, one sympathetic character for system design is that the eigenvalues of the closed-loop system matrix $A_c$ are insensitive to perturbations,
which leads to the following {\em state-feedback robust pole assignment problem}:

\noindent{\bf State-Feedback Robust Pole Assignment Problem (SFRPA)}
Find a solution $F\in\mathbb{R}^{m\times n}$ to the {\bf SFPA}, such that the closed-loop system is robust, that is, the eigenvalues of $A_c$ are as insensitive to perturbations on $A_c$ as possible.

The key to solve the {\bf SFRPA} is to choose an appropriate measure of robustness formulated in quantitative form.
Some measures can be found in \cite{XU, KNV, BN, Chu2, Dic}, such as the \texttt{condition number measurement}
$\kappa_F(X)=\|X\|_F\|X^{-1}\|_F$, where $X$ is the eigenvector matrix of $A_c$,
the \texttt{departure from normality} $\Delta_F(A_c)=\sqrt{\|A_c\|_F^2-\sum_{j=1}^{j=n}|\lambda_j|^2}$ and so on.
Ramar and Gourishankar \cite{RG} made an early contribution to the {\bf SFRPA}
and since then various optimization methods have been proposed based on different
measures \cite{ BN, CB, Chu2, KNV, Dic, Tits, LY}.
The most classic methods should be those proposed by Kautsky, Nichols and Van Dooren in \cite{KNV},
where $\kappa_F(X)$ is used as the  measure of robustness of the closed-loop system matrix.
However, Method $0$ in \cite{KNV} may fail to converge, Method $1$ may suffer from slow convergence,
and Method $2/3$ may not perform well on ill-conditioned problems.
Based on Method $0$ in \cite{KNV}, Tits and Yang \cite{Tits} proposed a method for solving the {\bf SFRPA}
by trying to maximize the absolute value of the determinant of the eigenvector matrix $X$.
The optimization processes are iterative, and hence generally expensive.
Recently, Chu \cite{Chu2} put forward a Schur-type method for the {\bf SFRPA}
by tending to minimize the departure from normality of the closed-loop system matrix $A_c$ via the Schur decomposition of $A_c$.
It computes the matrices $X$ and $T$ column by column, where $A_c=XTX^{-1}$, $X,T$ are real  and $T$ is upper  quasi-triangular,
such that the strictly block upper triangular elements of matrix $T$ are minimized in each step.
If $\lambda_1,\dots,\lambda_n$ are all real, \textbf{SCHUR} \cite{Chu2}
will generate an orthogonal matrix $X$, that is, $A_c=XTX^{-1}$ is the Schur decomposition of $A_c$.
This implies that the departures from normality of $A_c$ and $T$ are the same.
Hence the strategy aiming to minimize the departure from normality of $T$ is also pliable to $A_c$.
However, in case of complex conjugate  poles, it cannot produce an orthogonal $X$,
suggesting that the departure from normality of $A_c$ is generally not identical to that of $T$.
Hence, although it attempts to optimize the departure from normality of $T$, that of $A_c$ may still be large.

In this paper, we  propose a modified Schur method upon \textbf{SCHUR} \cite{Chu2},
where poles are assigned via the real Schur decomposition of $A_c=XTX^{\top}$,
with $X$ being real orthogonal and $T$ being real upper  quasi-triangular.
In each step (assigning a real pole or a pair of conjugate poles),
one optimization problem arises for purpose of minimizing the departure from normality of $T$.
When assigning a real pole, we improve the efficiency of \textbf{SCHUR} by computing the  SVD of a matrix,
instead of computing the GSVD of a matrix pencil.
When assigning a pair of conjugate poles,
by exploring the properties of the posed optimization problem,
we provide a polished way to obtain its suboptimal solution.
Numerical examples show that our method outperforms \textbf{SCHUR} when complex conjugate poles are involved.
We also compare our method with the MATLAB functions \textbf{place} (an implementation of Method 1 in \cite{KNV}),
\textbf{robpole} (an implementation of the method in \cite{Tits}) and the \textbf{O-SCHUR} algorithm
(an implementation of an optimization method in \cite{Chu2}) on some benchmark examples and randomly generated examples,
where numerical results show that our method is comparable in accuracy and robustness, while with lower computational costs.

The paper is organized as follows. In Section 2, we give some preliminaries which will be used
in subsequent sections. Our method is developed in Section 3,
including both the real case and the complex conjugate case.
Numerical results  are presented in Section 4.
Some concluding remarks are finally drawn in Section 5.

\section{Preliminaries}
In this section, we briefly review the parametric solutions to the {\bf SFPA}, and the departure from normality.

\subsection{Solutions to the SFPA}
The parametric solutions to the {\bf SFPA} can be expressed in several ways.
In this paper, as in \cite{Chu2}, we  formulate it
by using the real Schur decomposition of $A_c=A+BF$. Assume that the real Schur decomposition of $A+BF$ is
\begin{equation}\label{eqrealschur}
A + BF = XTX^{\top},
\end{equation}
where $X\in \mathbb{R}^{n\times n}$ is orthogonal, $T\in \mathbb{R}^{n\times n}$ is upper  quasi-triangular with only $1\times 1$ and $2\times2$ diagonal blocks.

Without loss of generality, we may assume that $B$ is of full column rank.
Let
\begin{equation}\label{eqqrofb}
B=Q\begin{bmatrix}R\\ 0\end{bmatrix}=\begin{bmatrix}Q_1&Q_2\end{bmatrix}\begin{bmatrix}R\\ 0\end{bmatrix}=Q_1R
\end{equation}
be the QR decomposition of $B$,
where $Q\in\mathbb{R}^{n\times n}$ is orthogonal,
$Q_1\in \mathbb{R}^{n\times m}$, and $R\in\mathbb{R}^{m\times m}$ is nonsingular and upper triangular.

It follows from \eqref{eqrealschur} that
\begin{equation}\label{eqchangingofrealschur}
AX + BFX - XT =0.
\end{equation}
Pre-multiplying \eqref{eqchangingofrealschur} by
$\diag(R^{-1}, I_{n-m})\begin{bmatrix}Q_1&Q_2\end{bmatrix}^{\top}$
on both sides gives
\begin{equation}\label{eqsolve}
   \left\{ \begin{array}{l}
   R^{-1}Q_1^{\top}AX + FX - R^{-1}Q_1^{\top}XT=0, \\
   Q_2^{\top}(AX -XT)=0.
   \end{array}\right.
\end{equation}
Consequently, if we get an orthogonal matrix $X$ and an upper  quasi-triangular matrix $T$ from the second equation of \eqref{eqsolve}, then a solution $F$ to the {\bf SFPA} will be
obtained immediately from the first equation of \eqref{eqsolve} as
\begin{equation} \label{eqsolveoff}
F=R^{-1}Q_1^{\top}(XTX^{\top}-A).
\end{equation}

\subsection{Departure from normality}

In this paper, we adopt the departure from normality of $A_c=A+BF$
as a  measure of robustness of the closed-loop system matrix as in \cite{Chu2},
which is defined as (\cite{Hen, SSun})
\[
\Delta_F(A_c)=\sqrt{\|A_c\|_F^2-\sum_{j=1}^{n}|\lambda_j|^2},
\]
where $\lambda_1,\dots,\lambda_n$ are the poles to be assigned, and hence eigenvalues of $A_c$.
Now let $D$ be the block diagonal part of $T$ with only $1\times 1$ and $2\times 2$ blocks on its diagonal.
Each $1\times1$ block of $D$ admits a real eigenvalue $\lambda_j$ of $T$,
while each $2\times 2$ block of $D$ admits a pair of conjugate eigenvalues
$\lambda_{j}=\alpha_j+ i \beta_j, \lambda_{j+1}=\bar{\lambda}_j$ and is of the form
$D_{j}=\begin{bmatrix}\begin{smallmatrix}\alpha_j&\delta_j\beta_j\\-\frac{\beta_j}{\delta_j}&\alpha_j
\end{smallmatrix}\end{bmatrix}\in \mathbb{R}^{2\times2}$ with $\delta_j\beta_j\ne 0$,
where $\delta_j$ is some real number.
Let $N=T-D=\begin{bmatrix}\breve{v}_1&\breve{v}_2&\cdots&\breve{v}_n\end{bmatrix}$ be the strictly upper  quasi-triangular part of $T$
with $\breve{v}_k=\begin{bmatrix}v_k^\top&0\end{bmatrix}^\top, v_k\in \mathbb{R}^{k-1} \text{or} \ \mathbb{R}^{k-2} $.
Direct calculations give rise to
\begin{equation}\label{departure}
\Delta_F^2(A_c)=\Delta_F^2(T)=\|N\|_F^2+\sum_{j}(\delta_j-\frac{1}{\delta_j})^2\beta_j^2,
\end{equation}
where the summation is over all $2\times 2$ blocks of $D$.

When all poles $\lambda_1,\dots,\lambda_n$ are real, the second part of $\Delta_{F}^2(A_c)$ in \eqref{departure} will vanish.
However, when some poles are non-real, not only the strictly block upper triangular part $N$ contributes to the departure from normality,
but also the block diagonal part $D$.
When some $|\delta_j|$ is large or close to zero, the second term can be pretty large,
which means that it is not negligible.

\section{Solving the SFRPA via the real Schur form}
In this section, we solve the {\bf SFRPA} by finding an orthogonal matrix
$X=\begin{bmatrix}x_1&x_2&\cdots&x_{n}\end{bmatrix}$
and an upper  quasi-triangular matrix $T=D+N$ satisfying
the second equation of \eqref{eqsolve},
such that $\Delta_{F}^2(A_c)$ in \eqref{departure} is minimized.
Obtaining a global optimization solution to the problem
$\min\{\Delta_F^2(A_c)\}$ is rather difficult.
In this paper, we  propose an efficient method to get a suboptimal solution,
which balances the contributions of $N$ and $D$ to the departure from normality.
As in \cite{Chu2}, we compute the matrices $X$ and $T$ column by column.
%
%Here are two notations first. We denote the space spanned by the columns of the
%matrix $S$ as $\mathcal{R}(S)$, and the null space of the matrix $M$ as $\mathcal{N}(M)$.

For any matrix $S$, we denote its range space and null space by $\mathcal{R}(S)$ and $\mathcal{N}(S)$, respectively.
Assume that we have already obtained
$X_{j}=\begin{bmatrix}x_1&x_2&\cdots&x_{j}\end{bmatrix} \in\mathbb{R}^{n\times j}$ and
$T_{j} \in\mathbb{R}^{j\times j}$ satisfying
\begin{align}\label{eq1}
Q_2^\top (AX_{j} - X_{j}T_{j})=0,\qquad X_{j}^\top X_{j}=I_{j},
\end{align}
where $T_j$ is upper  quasi-triangular and
$\lambda(T_j)=\{\lambda_k\}_{k=1}^{k=j}$.
We then are to assign the pole $\lambda_{j+1}$ (if $\lambda_{j+1}$ is real)
or poles $\lambda_{j+1},\bar{\lambda}_{j+1}$ (if $\lambda_{j+1}$ is non-real)
to get $x_{j+1}$, $\breve{v}_{j+1}$ or $x_{j+1}, x_{j+2}$, $\breve{v}_{j+1}, \breve{v}_{j+2}$,
such that the departure from normality of $A_c$ is optimized in some sense.
This procedure is repeated until all columns of $X$ and $T$ are acquired,
and eventually a solution $F$ to the {\bf SFRPA} would be computed from \eqref{eqsolveoff}.
In the following subsections we will distinguish two different cases when $\lambda_{j+1}$ is real or non-real.

Before this, we shall show how to get the first one (two) column(s) of $X$ and $T$.
If $\lambda_1$ is real, the first column of $T$ is then $\lambda_1e_1$, or $T_1=\lambda_1$,
and the first column $x_1$ of $X$ must satisfy
\begin{equation}\label{initial}
Q_2^{\top}(A-\lambda_1I_n)x_1=0,
\end{equation}
and $\|x_1\|_2=1$. Let the columns of $S\in\mathbb{R}^{n\times r}$
be an orthonormal basis of $\mathcal{N}(Q_2^{\top}(A-\lambda_1I_n))$,
then $x_1$ can be chosen to be any unit vector in $\mathcal{R}(S)$. We take
\begin{align}\label{x1real}
x_1=(S\begin{bmatrix}1& \ldots&1\end{bmatrix}^{\top})/{\|S\begin{bmatrix}1& \ldots&1\end{bmatrix}^{\top}\|_2}
\end{align}
in our algorithm as in \cite{Chu2}, and then initially set $X_1=x_1,T_1=\lambda_1$.

If $\lambda_1=\alpha_1+i\beta_1$ is non-real, to get the real Schur form,
we should place $\bar{\lambda}_1=\alpha_1-i\beta_1$ together with $\lambda_1$.
Notice that $T_2$ is of the form
$T_2=\begin{bmatrix}\begin{smallmatrix}\alpha_{1} & \delta_1 \beta_{1} \\ -\beta_{1}/\delta_1 & \alpha_{1}\end{smallmatrix}\end{bmatrix}$ with $0\ne\delta_1\in\mathbb{R}$,
then the first two columns $x_1,x_2\in\mathbb{R}^n$ of $X$ should be chosen to satisfy
\begin{align}\label{x1x2}
Q_2^{\top}(A\begin{bmatrix}x_1 &x_2\end{bmatrix}-\begin{bmatrix}x_1 &x_2\end{bmatrix}T_2)=0, \quad x_1^\top x_2=0, \quad \|x_1\|_2=\|x_2\|_2=1,
\end{align}
so that $(\delta_1-\frac{1}{\delta_1})^2\beta_1^2$ is minimized,
which obviously achieves its minimum when $\delta_1=1$.
Let the columns of $S\in\mathbb{C}^{n\times r}$ be an orthonormal basis of $\mathcal{N}(Q_2^{\top}(A-\lambda_1I_n))$,
and $S_R=\mbox{Re}(S)$, $S_I=\mbox{Im}(S)$.  Direct calculations show that such $x_1,x_2$ satisfying \eqref{x1x2} with $\delta_1=1$ can be obtained by
\begin{align}\label{x1x2get}
x_1&=\begin{bmatrix}S_{R}& -S_I\end{bmatrix}\begin{bmatrix}\gamma_1&\ldots& \gamma_r& \zeta_1& \ldots& \zeta_r\end{bmatrix}^{\top}, \qquad x_2&=\begin{bmatrix}S_I& S_R\end{bmatrix}\begin{bmatrix}\gamma_1&\ldots& \gamma_r& \zeta_1& \ldots& \zeta_r\end{bmatrix}^{\top},
\end{align}
with $x_1^\top x_2=0$ and $\|x_1\|_2=\|x_2\|_2=1$.
Clearly,
\begin{equation}\label{initial_com}
\begin{split}
&x_1^{\top}x_2+x_2^{\top}x_1\\
=&\begin{bmatrix}\gamma_1&\ldots& \gamma_r&\zeta_1& \ldots& \zeta_r\end{bmatrix}
\begin{bmatrix}S_R^{\top}S_I+S_I^{\top}S_R&S_R^{\top}S_R-S_I^{\top}S_I\\
S_R^{\top}S_R-S_I^{\top}S_I&-(S_R^{\top}S_I+S_I^{\top}S_R) \end{bmatrix}
\begin{bmatrix}\gamma_1&\ldots& \gamma_r& \zeta_1& \ldots& \zeta_r\end{bmatrix}^{\top},\\
&x_1^{\top}x_1-x_2^{\top}x_2\\
=&\begin{bmatrix}\gamma_1&\ldots& \gamma_r& \zeta_1& \ldots& \zeta_r\end{bmatrix}
\begin{bmatrix}S_R^{\top}S_R-S_I^{\top}S_I&-(S_R^{\top}S_I+S_I^{\top}S_R)\\
-(S_R^{\top}S_I+S_I^{\top}S_R)&S_I^{\top}S_I-S_R^{\top}S_R \end{bmatrix}
\begin{bmatrix}\gamma_1&\ldots& \gamma_r& \zeta_1& \ldots& \zeta_r\end{bmatrix}^{\top}.
\end{split}
\end{equation}
Note that the two matrices in the above two equations are symmetric Hamiltonian systems owning special
properties. So we exhibit some simple results about symmetric Hamiltonian system which will be used here and
when assigning the complex conjugate poles. Both results can be verified directly, and we omit the proof.
\begin{Lemma}\label{Lemma3.1}
Let $A, B\in \mathbb{R}^{n\times n}$ satisfying $A^\top=A, B^\top=B.$ If $\lambda$ is an eigenvalue of
$\begin{bmatrix}A&B\\B&-A\end{bmatrix}$ and $\begin{bmatrix}x^\top&y^\top\end{bmatrix}^\top$ is the corresponding eigenvector,
then
\begin{align*}
\begin{bmatrix}A&B\\B&-A\end{bmatrix} \begin{bmatrix}x&-y\\y&x\end{bmatrix}=\begin{bmatrix}x&-y\\y&x\end{bmatrix}
\begin{bmatrix}\lambda& \\ & -\lambda\end{bmatrix},
\end{align*}
and
\begin{align*}
\begin{bmatrix}B&-A\\-A&-B\end{bmatrix}\begin{bmatrix}x&-y\\y&x\end{bmatrix}
\begin{bmatrix}\frac{\sqrt{2}}{2}&-\frac{\sqrt{2}}{2}\\-\frac{\sqrt{2}}{2}&-\frac{\sqrt{2}}{2}\end{bmatrix}
=\begin{bmatrix}x&-y\\y&x\end{bmatrix}
\begin{bmatrix}\frac{\sqrt{2}}{2}&-\frac{\sqrt{2}}{2}\\-\frac{\sqrt{2}}{2}&-\frac{\sqrt{2}}{2}\end{bmatrix}
\begin{bmatrix}\lambda& \\ & -\lambda\end{bmatrix}.
\end{align*}
\end{Lemma}

\begin{Lemma}(Property of Two Hamiltonian Systems)\label{Lemma3.2}
Let $A, B\in \mathbb{R}^{n\times n}$ be symmetric, and let
$\begin{bmatrix}A&B\\B&-A\end{bmatrix}=U\diag(\Theta, -\Theta) U^\top$ be the spectral decomposition, where $\Theta=\diag(\theta_1, \theta_2, \ldots, \theta_n)$ with
$\theta_1\geq \theta_2\geq\ldots \geq\theta_n\geq 0$.
If the $j$-th column $u_{j}$  and the $(n+j)$-th column $u_{n+j}$
of $U$ satisfy $u_{n+j}=\begin{bmatrix}&-I_n\\I_n&\end{bmatrix}u_{j}$,
then $\begin{bmatrix}B&-A\\-A&-B\end{bmatrix}=U\begin{bmatrix}0&-\Theta\\-\Theta&0\end{bmatrix}U^\top$.
\end{Lemma}
Applying Lemma \ref{Lemma3.2} to the two symmetric Hamiltonian systems which appeared in \eqref{initial_com}, that is
\begin{align*}
\begin{bmatrix}S_R^{\top}S_I+S_I^{\top}S_R&S_R^{\top}S_R-S_I^{\top}S_I\\
S_R^{\top}S_R-S_I^{\top}S_I&-(S_R^{\top}S_I+S_I^{\top}S_R) \end{bmatrix}=&U\diag(\Theta, -\Theta) U^\top,\\
\begin{bmatrix}S_R^{\top}S_R-S_I^{\top}S_I&-(S_R^{\top}S_I+S_I^{\top}S_R)\\
-(S_R^{\top}S_I+S_I^{\top}S_R)&S_I^{\top}S_I-S_R^{\top}S_R \end{bmatrix}=&U\begin{bmatrix}0&-\Theta\\-\Theta&0\end{bmatrix}U^{\top},
\end{align*}
then if we let
\begin{align}\label{gammamunu}
\begin{bmatrix}\gamma_1&\ldots& \gamma_r&\zeta_1& \ldots& \zeta_r\end{bmatrix}^\top=U
\begin{bmatrix}\mu_1&\ldots&\mu_r& \nu_1&\ldots&\nu_r\end{bmatrix}^{\top},
\end{align}
$x_1^{\top}x_2+x_2^{\top}x_1=\sum_{j=1}^r\theta_j(\mu_j^2-\nu_j^2)$ and
$x_1^{\top}x_1-x_2^{\top}x_2=-2\sum_{j=1}^r\theta_j\mu_j\nu_j$ follow. Without loss of generality, we may assume that
$\theta_1\geq\theta_2\geq\ldots\geq\theta_r\geq0$, then by taking
\begin{subequations}\label{munuinitial}
\begin{align}
&\mu_3=\nu_3=\ldots=\mu_r=\nu_r=0, \quad \mu_1=-\nu_1=\sqrt{\frac{\theta_2}{\theta_1}\mu_2^2}, \\ &\mu_2=\nu_2=\frac{1}{\|\begin{bmatrix}S_R&-S_I\end{bmatrix}U
\begin{bmatrix}\sqrt{\frac{\theta_2}{\theta_1}}&1&0&\cdots&0&-\sqrt{\frac{\theta_2}{\theta_1}}&1&0&\cdots&0\end{bmatrix}^{\top}\|_2},
\end{align}
\end{subequations}
it is easy to verify that  \eqref{x1x2} holds with $x_1$ and $x_2$ computed by \eqref{x1x2get} and \eqref{gammamunu}.
Hence, we can still choose initial vectors $x_1$ and $x_2$, so
that $(\delta_1-\frac{1}{\delta_1})^2\beta_1^2=0$. We then initially set
\begin{align}\label{initialnonreal}
X_2=\begin{bmatrix}x_1 &x_2\end{bmatrix},\qquad T_2=\begin{bmatrix}\alpha_{1} & \beta_{1} \\ -\beta_{1}& \alpha_{1}\end{bmatrix}.
\end{align}

Now assume that \eqref{eq1} has been satisfied with $j\geq1$,
we shall then assign the next pole $\lambda_{j+1}$.
%The case that $\lambda_{j+1}$ is real is discussed in the following subsection,
%and  that $\lambda_{j+1}$ is non-real will be disposed in the next subsection.

\subsection{Assigning a real pole}
Assume that $\lambda_{j+1}$ is real,
then the $(j+1)$-th diagonal element of $T$ must be $\lambda_{j+1}$.
Comparing the $(j+1)$-th column of $Q_2^{\top}AX - Q_2^{\top}XT=0$ gives rise to
\begin{equation}
Q_2^\top Ax_{j+1} - Q_2^\top X_{j}v_{j+1}- \lambda_{j+1}Q_2^\top x_{j+1}=0.
\end{equation}
Recall the definition of the departure from normality of $A_c$ in \eqref{departure}
and notice that we are now computing the $(j+1)$-th columns of $X$ and $T$,
it is then natural to consider the following optimization problem:
%\begin{gather}\label{eqreal-opt}
%\min_{\|x_{j+1}\|_2=1}\|v_{j+1}\|_2^2\\
%\mbox{s.t.}\left\{ \begin{array}{l}
% Q_2^\top Ax_{j+1}-\lambda_{j+1} Q_2^\top x_{j+1}-Q_2^\top X_{j}v_{j+1}=0,\\
% X_j^\top x_{j+1}=0,\end{array}
%\right.
%\end{gather}
%which can be rewritten as
\begin{gather}\label{eqreal-opt-equal}
\min_{\|x_{j+1}\|_2=1}\|v_{j+1}\|_2^2\\
\mbox{s.t. } M_{j+1}
\begin{bmatrix}x_{j+1}\\v_{j+1}\end{bmatrix}=0 ,\label{eqreal-opt-constrain}
\end{gather}
where
\begin{align}\label{M}
M_{j+1}=\begin{bmatrix}Q_2^\top(A-\lambda_{j+1}I_n)& -Q_2^\top X_j\\X_j^\top&0\end{bmatrix}.
\end{align}
Let $r=\dim\mathcal{N}(M_{j+1})$.
Then it follows from the controllability of $(A, B)$
that $Q_2^\top(A-\lambda_{j+1} I_n)$ is of full row rank,
indicating that  $n-m\leq \rank (M_{j+1})\leq n-m+j$ and $\mathcal{N}(M_{j+1})\neq\emptyset$ (\cite{Chu2}).
Suppose that the columns of
$S=\begin{bmatrix}S_{1}^\top &S_{2}^\top\end{bmatrix}^\top$
with $S_{1}\in\mathbb{R}^{n\times r}, S_{2}\in\mathbb{R}^{j\times r}$
form an orthonormal basis of $\mathcal{N}(M_{j+1})$, then
\eqref{eqreal-opt-constrain} shows that
\begin{align}\label{eqrealxv}
\begin{array}{lll}
x_{j+1}=S_{1}y,& v_{j+1}=S_{2}y, &\qquad  \forall y\in \mathbb{R}^r.
\end{array}
\end{align}
Consequently, the optimization problem \eqref{eqreal-opt-equal} subject to
\eqref{eqreal-opt-constrain} is equivalent to the following problem:
\begin{equation}\label{eqreal-opt-equal-2}
\min_{y^\top S_{1}^\top S_{1}y=1}y^\top S_{2}^\top S_{2}y.
\end{equation}

Perceived that the discussions above can also be found in \cite{Chu2},
and the constrained optimization problem \eqref{eqreal-opt-equal-2}
is solved via the GSVD of the matrix pencil $(S_1,S_2)$.
We put forward a simpler approach here. Actually, since $S^\top S=I_r$, we have
$S_{2}^\top S_{2}=I_r-S_{1}^\top S_{1}$.
Thus the problem \eqref{eqreal-opt-equal-2} is equivalent to
\begin{equation}\label{eqreal-opt-equal-3}
\min_{y^\top S_{1}^\top S_{1}y=1}y^\top y,
\end{equation}
whose minimum value is acquired when $y$ is an eigenvector of $S_1^{\top}S_1$
corresponding to its greatest eigenvalue and satisfies $y^\top S_{1}^\top S_{1}y=1$.
Once such $y$ is obtained, $x_{j+1}$ and $v_{j+1}$ can be given by \eqref{eqrealxv}.
We may then update $X_j$ and $T_j$ as
\begin{align}\label{updatereal}
X_{j+1}=\begin{bmatrix}X_j &x_{j+1}\end{bmatrix}\in\mathbb{R}^{n\times (j+1)},\qquad
T_{j+1}=\begin{bmatrix}T_j &v_{j+1}\\0 &\lambda_{j+1}\end{bmatrix}\in\mathbb{R}^{(j+1)\times (j+1)},
\end{align}
and continue with the next pole $\lambda_{j+2}$.

\subsection{Assigning a pair of conjugate poles }
In this subsection, we will consider the case that $\lambda_{j+1}$ is non-real. To obtain a real matrix $F$ from the real Schur form of $A_c=A+BF$, we would assign $\lambda_{j+1}$ and $\lambda_{j+2}=\bar{\lambda}_{j+1}$ simultaneously
to get the $(j+1)$-th and $(j+2)$-th columns of $X$ and $T$.

\subsubsection{Initial optimization problem}
Assume that $\lambda_{j+1}=\alpha_{j+1} + i \beta_{j+1} \,\,(\beta_{j+1}\ne 0)$ and let $D_{\delta}=\begin{bmatrix}\alpha_{j+1} & \delta \beta_{j+1} \\ -\beta_{j+1}/\delta & \alpha_{j+1}\end{bmatrix}$ be the diagonal block in $T$ whose eigenvalues are $\lambda_{j+1}$ and $\bar{\lambda}_{j+1}$. By comparing the $(j+1)$-th and $(j+2)$-th columns of $Q_2^{\top}AX - Q_2^{\top}XT=0$, we have
\begin{equation}
Q_2^\top A \begin{bmatrix}x_{j+1}&x_{j+2}\end{bmatrix} -
Q_2^\top X_{j}\begin{bmatrix}v_{j+1}&v_{j+2}\end{bmatrix}- Q_2^\top
\begin{bmatrix}x_{j+1}&x_{j+2}\end{bmatrix} D_{\delta}=0.
\end{equation}
Recalling the form of $\Delta_F^2(A_c)$ in \eqref{departure}, it is then natural to
consider the following optimization problem:
\begin{subequations}\label{opt1}
\begin{align}
\min_{\delta, v_{j+1}, v_{j+2}}&\|v_{j+1}\|_2^2 + \|v_{j+2}\|_2^2 + \beta_{j+1}^2 (\delta -\frac{1}{\delta})^2\label{eqcomplex-opt}\\
\mbox{s.t.}\quad &Q_2^\top (A \begin{bmatrix}x_{j+1}& x_{j+2}\end{bmatrix}
-X_j \begin{bmatrix}v_{j+1} &v_{j+2}\end{bmatrix} - \begin{bmatrix}x_{j+1}& x_{j+2}\end{bmatrix}D_{\delta})=0,\label{eqcomplex-constraina}\\
&X_j^\top \begin{bmatrix}x_{j+1}& x_{j+2}\end{bmatrix}=0,\label{eqcomplex-constrainb}\\
&\begin{bmatrix}x_{j+1}& x_{j+2}\end{bmatrix}^\top \begin{bmatrix}x_{j+1}& x_{j+2}\end{bmatrix}=I_2.\label{eqcomplex-constrainc}
\end{align}
\end{subequations}
The constraints \eqref{eqcomplex-constraina} and \eqref{eqcomplex-constrainc} are nonlinear. In \cite{Chu2}, the author solves this optimization problem by taking $\delta=1$ and neglecting the orthogonal requirement  $x_{j+1}^\top x_{j+2}=0$. These simplify the problem significantly. However, it cannot lead to the real Schur form of the
closed-loop system matrix $A_c$, since $x_{j+1}$ is generally not orthogonal to $x_{j+2}$. Moreover,
the minimum value of the simplified optimization problem in \cite{Chu2} may be much greater than that of the original problem \eqref{opt1}.

We may rewrite the optimization problem \eqref{opt1} into another equivalent form. If we write $\delta=\frac{\delta_2}{\delta_1}$ with $\delta_1,\delta_2>0$, and set
$D_0=\begin{bmatrix}\begin{smallmatrix}\alpha_{j+1} & \beta_{j+1} \\ -\beta_{j+1} & \alpha_{j+1}\end{smallmatrix}\end{bmatrix}$, then $D_{\delta}=\begin{bmatrix}\begin{smallmatrix}1/{\delta_1}&\\&1/{\delta_2}\end{smallmatrix}\end{bmatrix} D_0
\begin{bmatrix}\begin{smallmatrix}\delta_1&\\&\delta_2\end{smallmatrix}\end{bmatrix}$. Redefine
${x}_{j+1}\triangleq\frac{x_{j+1}}{\delta_1}, {x}_{j+2}\triangleq\frac{x_{j+2}}{\delta_2},
{v}_{j+1}\triangleq\frac{v_{j+1}}{\delta_1}, {v}_{j+2}\triangleq\frac{v_{j+2}}{\delta_2}$,
then the optimization problem \eqref{opt1} is equivalent to
\begin{subequations}\label{opt2}
\begin{align}
\min_{\delta_1, \delta_2, v_{j+1}, v_{j+2}}&\|\delta_1v_{j+1}\|_2^2 + \|\delta_2v_{j+2}\|_2^2 + \beta_{j+1}^2 (\frac{\delta_1}{\delta_2} -\frac{\delta_2}{\delta_1})^2\label{eqcomplex-opt-2}\\
\mbox{s.t.}\quad &Q_2^\top (A \begin{bmatrix}x_{j+1}& x_{j+2}\end{bmatrix}
-X_j \begin{bmatrix}v_{j+1} &v_{j+2}\end{bmatrix} - \begin{bmatrix}x_{j+1}& x_{j+2}\end{bmatrix}D_0)=0,\label{eqcomplex-opt-2-constraina}\\
&X_j^\top \begin{bmatrix}x_{j+1}& x_{j+2}\end{bmatrix}=0,\label{eqcomplex-opt-2-constrainb}\\
&\begin{bmatrix}x_{j+1}& x_{j+2}\end{bmatrix}^\top \begin{bmatrix}x_{j+1}& x_{j+2}\end{bmatrix}=\begin{bmatrix}1/{\delta_1^2}&\\&1/{\delta_2^2}\end{bmatrix}.
\label{eqcomplex-opt-2-constrainc}
\end{align}
\end{subequations}
Here the constraint \eqref{eqcomplex-opt-2-constraina} becomes linear.
Once a solution to the optimization problem \eqref{opt2} is obtained, we
need to redefine
\[
v_{j+1}\triangleq\frac{v_{j+1}}{\|x_{j+1}\|_2},\quad v_{j+2}\triangleq\frac{v_{j+2}}{\|x_{j+2}\|_2},\quad x_{j+1}\triangleq\frac{x_{j+1}}{\|x_{j+1}\|_2},\quad x_{j+2}\triangleq\frac{x_{j+2}}{\|x_{j+2}\|_2}
\]
as the corresponding columns of $T$ and $X$.

The constraints \eqref{eqcomplex-opt-2-constraina} and \eqref{eqcomplex-opt-2-constrainb} are linear. 
Actually, all vectors $x_{j+1},x_{j+2},v_{j+1},v_{j+2}$ satisfying these two constraints can be found via the null space of the matrix
\begin{align}\label{Mcomplex}
M_{j+1}=\begin{bmatrix}Q_2^{\top}(A-(\alpha_{j+1}+i\beta_{j+1})I_n)& -Q_2^{\top}X_j\\ X_j^{\top}&0\end{bmatrix}.
\end{align}
Specifically, for any $x_{j+1},x_{j+2},v_{j+1},v_{j+2}$ satisfying \eqref{eqcomplex-opt-2-constraina} and \eqref{eqcomplex-opt-2-constrainb}, direct calculations show that $M_{j+1}\begin{bmatrix}x_{j+1}+i x_{j+2}\\v_{j+1}+i v_{j+2}\end{bmatrix}=0$. Conversely, for any vector $\begin{bmatrix}z^{\top}&w^{\top}\end{bmatrix}^{\top}\in\mathcal{N}(M_{j+1})$, the vectors $x_{j+1}=\mbox{Re}(z), x_{j+2}=\mbox{Im}(z), v_{j+1}=\mbox{Re}(w), v_{j+2}=\mbox{Im}(w)$ satisfy \eqref{eqcomplex-opt-2-constraina} and \eqref{eqcomplex-opt-2-constrainb}.
The constraint \eqref{eqcomplex-opt-2-constrainc} shows that $x_{j+1}^\top x_{j+2}=0$. For any vector $\begin{bmatrix}z^{\top}&w^{\top}\end{bmatrix}^{\top}\in\mathcal{N}(M_{j+1})$ with $\mbox{Re}(z)$ and $\mbox{Im}(z)$ being linearly independent, we may then
orthogonalize $\mbox{Re}(z)$ and $\mbox{Im}(z)$ by the Jacobi transformation as follows to get $x_{j+1}$ and $x_{j+2}$ satisfying $x_{j+1}^\top x_{j+2}=0$.
Let $\varrho_1=\|\mbox{Re}(z)\|_2^2,\ \varrho_2=\|\mbox{Im}(z)\|_2^2, \ \gamma=\mbox{Re}(z)^{\top}\mbox{Im}(z)$ and $\tau=\frac{\varrho_2-\varrho_1}{2\gamma}$, and define $t$ as
\begin{displaymath}
t=\left\{ \begin{array}{ll}
1/(\tau+\sqrt{1+\tau^2}), & \text{ if}\quad \tau\geq0,\\
-1/(-\tau+\sqrt{1+\tau^2}), & \text{ if}\quad \tau<0.\\
\end{array} \right.
\end{displaymath}
Let $c=1/\sqrt{1+t^2}$, $s=tc$. Then $x_{j+1}$ and $x_{j+2}$ obtained by
\begin{align} \label{Jocobi_x}
\begin{bmatrix}x_{j+1}& x_{j+2} \end{bmatrix}=\begin{bmatrix}\mbox{Re}(z)& \mbox{Im}(z) \end{bmatrix}
\begin{bmatrix}c & s\\ -s&c\end{bmatrix}
\end{align}
satisfy $x_{j+1}^\top x_{j+2}=0$. Moreover, if we let
\begin{align}\label{Jocobi_v}
\begin{bmatrix}v_{j+1}&v_{j+2}\end{bmatrix}=\begin{bmatrix}\mbox{Re}(w)& \mbox{Im}(w) \end{bmatrix}
\begin{bmatrix}c & s\\ -s&c\end{bmatrix},
\end{align}
then $x_{j+1},x_{j+2},v_{j+1},v_{j+2}$ satisfy \eqref{eqcomplex-opt-2-constraina} and \eqref{eqcomplex-opt-2-constrainb}. Hence,
we can get $x_{j+1},x_{j+2},v_{j+1},v_{j+2}$ satisfying the constrains \eqref{eqcomplex-opt-2-constraina}-\eqref{eqcomplex-opt-2-constrainc} in this way. Furthermore,
\begin{equation} \label{eqortho-one}
1/{\delta_1^2}=\|x_{j+1}\|_2^2=\|x\|_2^2-\omega,\quad
1/{\delta_2^2}=\|x_{j+2}\|_2^2=\|y\|_2^2+\omega,
\end{equation}
where $x=\mbox{Re}(z)$, $y=\mbox{Im}(z)$, $\omega=\frac{2(x^{\top}y)^2}{\|y\|_2^2-\|x\|_2^2+\sqrt{4(x^{\top}y)^2+(\|y\|_2^2-\|x\|_2^2)^2}}$ if $\|x\|_2<\|y\|_2$; and $\omega=\frac{2(x^{\top}y)^2}{\|y\|_2^2-\|x\|_2^2-\sqrt{4(x^{\top}y)^2
+(\|y\|_2^2-\|x\|_2^2)^2}}$ if $\|x\|_2\ge\|y\|_2$.

\subsubsection{The suboptimal strategy}

It is hard to get an optimal solution to \eqref{opt2} since it is a nonlinear optimization problem with quadratic constraints. Even if
such an optimal solution can be found, the cost will be expensive. So instead of finding an optimal solution,
we prefer to get a suboptimal one with less computational cost.

Let the columns of
$S=\begin{bmatrix}S_1^{\top}&S_2^{\top}\end{bmatrix}^{\top}\in \mathbb{C}^{(n+j)\times r}$
with $S_1\in \mathbb{C}^{n\times r}$ and $S_2\in \mathbb{C}^{j\times r}$
form an orthonormal basis of $\mathcal{N}(M_{j+1})$,
and let $S_1=U\Sigma V^{*}$ be the SVD of $S_1$.
Since $S_1^{*}S_1+S_2^{*}S_2=I_r$, it follows that $S_2^{*}S_2=V(I_r-\Sigma^{*}\Sigma)V^{*}$.
For any vector $\begin{bmatrix}z^{\top}&w^{\top}\end{bmatrix}^{\top}\in \mathcal{N}(M_{j+1})$
with $z\in \mathbb{C}^{n}$ and $w\in \mathbb{C}^{j}$,
there exists $b\in \mathbb{C}^{r}$ such that $z=S_1b=U(\Sigma V^{*}b)$ and $w=S_2b$. Hence
\begin{align*}
%\label{length_z}
\|z\|_2\leq\sigma_1 \|b\|_2 \qquad  \text{ and } \qquad
\|w\|_2^2\geq (1-\sigma_1^2)\|b\|_2^2,
\end{align*}
where $\sigma_1$ is the largest singular value of $S_1$. Now suppose that the real part and the imaginary part of
$z$ are linearly independent satisfying $\|\mbox{Re}(z)\|_2\leq\|\mbox{Im}(z)\|_2$,
and $x_{j+1}, x_{j+2}$, $v_{j+1}, v_{j+2}$ are obtained from the the
Jacobi orthogonal process \eqref{Jocobi_x}, \eqref{Jocobi_v}. Define $C=\frac{\|z\|_2}{\|x_{j+1}\|_2}$, then $C\geq\sqrt{2}$ and the objective function in \eqref{eqcomplex-opt-2} becomes
\begin{equation}\label{value_cost}
\begin{split}
&\|\delta_1v_{j+1}\|_2^2 + \|\delta_2v_{j+2}\|_2^2 + \beta_{j+1}^2 (\frac{\delta_1}{\delta_2} -\frac{\delta_2}{\delta_1})^2\\
=&\frac{C^2}{C^2-1}\frac{\|w\|_2^2}{\|z\|_2^2}+\frac{C^4-2C^2}{C^2-1}\frac{\|v_{j+1}\|_2^2}{\|z\|_2^2}+\beta_{j+1}^2(C^2-3+\frac{1}{C^2-1}).
\end{split}
\end{equation}
%with $\delta_1=\frac{1}{\|x_{j+1}\|_2}, \delta_2=\frac{1}{\|x_{j+2}\|_2}$.
Obviously,
\begin{align}\label{value_cost_ieq}
\frac{C^2}{C^2-1}\frac{\|w\|_2^2}{\|z\|_2^2}\leq
\frac{C^2}{C^2-1}\frac{\|w\|_2^2}{\|z\|_2^2}+\frac{C^4-2C^2}{C^2-1}\frac{\|v_{j+1}\|_2^2}{\|z\|_2^2}\leq
C^2\frac{\|w\|_2^2}{\|z\|_2^2}.
\end{align}
So the objective function in \eqref{eqcomplex-opt-2} depends on $\frac{\|w\|_2^2}{\|z\|_2^2}$ and $C$ with
$\min\frac{\|w\|_2^2}{\|z\|_2^2}=\frac{1-\sigma_1^2}{\sigma_1^2}$. In our suboptimal strategy, we will first take
$b$ from $\subspan\{Ve_1\}$, where $e_i$ is the $i$-th column of the identity matrix. With this choice, $\frac{\|w\|_2^2}{\|z\|_2^2}$ achieves its minimum value. And the following theorem shows the relevant results.

\begin{Theorem}\label{Theorem3.1}
With the notations above, let $u_1$ be the first column of $U$ and assume that $\mbox{Re}(u_1)$ and $\mbox{Im}(u_1)$ are linearly independent. Let $x_{j+1}$ and $x_{j+2}$ be the vectors obtained from $\mbox{Re}(u_1)$ and $\mbox{Im}(u_1)$
via the Jacobi orthogonal process
\begin{align*}
\begin{bmatrix}x_{j+1}&x_{j+2}\end{bmatrix}=\begin{bmatrix}\mbox{Re}(u_1)&\mbox{Im}(u_1)\end{bmatrix}\begin{bmatrix}c&s\\-s&c\end{bmatrix},
\end{align*}
and let
\begin{align*}
\begin{bmatrix}v_{j+1}&v_{j+2}\end{bmatrix}=\begin{bmatrix}\mbox{Re}(w)&\mbox{Im}(w)\end{bmatrix}\begin{bmatrix}c&s\\-s&c\end{bmatrix},
\end{align*}
where $w=S_2Ve_1/\sigma_1$. Then $x_{j+1},x_{j+2},v_{j+1},v_{j+2}$ satisfy the constrains \eqref{eqcomplex-opt-2-constraina}-\eqref{eqcomplex-opt-2-constrainc},
%. Then $\min\{\|x_{j+1}\|_2, \|x_{j+2}\|_2\}\leq\frac{\sqrt{2}}{2}$
and the value of the corresponding objective function specified by \eqref{eqcomplex-opt-2} will be no larger than
\[\frac{1}{\min\{\|x_{j+1}\|_2^2, \|x_{j+2}\|_2^2\}}(\frac{1-\sigma_1^2}{\sigma_1^2}+\beta_{j+1}^2).\]
\end{Theorem}
\begin{proof}
The first part of the theorem is obvious. To prove the second part, note that
here $b=\frac{Ve_1}{\sigma_1}$, $\|z\|_2=\|u_1\|_2=1, \|w\|_2^2=\frac{1-\sigma_1^2}{\sigma_1^2}$. If $\|\mbox{Re}(u_1)\|_2\leq\|\mbox{Im}(u_1)\|_2$, it then follows directly from \eqref{value_cost}, \eqref{value_cost_ieq} and $C^2-3+\frac{1}{C^2-1}\leq C^2$ with $C=\frac{1}{\|x_{j+1}\|_2}$.
The case when $\|\mbox{Re}(u_1)\|_2\geq\|\mbox{Im}(u_1)\|_2$ can be proved similarly.
\end{proof}

%\begin{Remark}
%Note that the restriction on the intermediate result $x_{j+1}$ in the Theorem \ref{Theorem3.1}
%is essentially an implicit requirement to $\mbox{Re}(u_1)$ and $\mbox{Im}(u_1)$ due to %\eqref{eqortho-one}.
%\end{Remark}

Theorem \ref{Theorem3.1} shows that if $\mbox{Re}(u_1)$ and $\mbox{Im}(u_1)$ are
linearly independent, and $\min\{\|x_{j+1}\|_2, \|x_{j+2}\|_2\}$ is not pathologically small, the above procedure will generate $x_{j+1},x_{j+2},v_{j+1},v_{j+2}$ satisfying the constrains \eqref{eqcomplex-opt-2-constraina}-\eqref{eqcomplex-opt-2-constrainc}, and the value of the corresponding objective function in \eqref{eqcomplex-opt-2} is not too large. We then take these $x_{j+1},x_{j+2},v_{j+1},v_{j+2}$ as the suboptimal solution. However, if $\mbox{Re}(u_1)$ and $\mbox{Im}(u_1)$ are linearly dependent, we cannot get orthogonal $x_{j+1}$ and $x_{j+2}$ via the Jacobi orthogonal process. Even if $\mbox{Re}(u_1)$ and $\mbox{Im}(u_1)$ are
linearly independent, the resulted $\min\{\|x_{j+1}\|_2, \|x_{j+2}\|_2\}$ might
be fairly small, which means that the corresponding value of the objective function might be large.
In this case, we would choose $b$ from $\subspan\{Ve_1, Ve_2\}$.

Define
\begin{align}\label{tildexy}
&\tilde{x}_1+i\tilde{y}_1=z_1=u_1=\frac{S_1Ve_1}{\sigma_1}, &&w_1=\frac{S_2Ve_1}{\sigma_1},\notag\\  &\tilde{x}_2+i\tilde{y}_2=z_2=u_2=\frac{S_1Ve_2}{\sigma_2}, &&w_2=\frac{S_2Ve_2}{\sigma_2},
\end{align}
where $\sigma_1, \sigma_2$ are the first two greatest singular values of $S_1$. Let
$b=\begin{bmatrix}\begin{smallmatrix}\frac{Ve_1}{\sigma_1}&\ &\frac{Ve_2}{\sigma_2}\end{smallmatrix}\end{bmatrix}
\begin{bmatrix}\begin{smallmatrix}\gamma_1+i\zeta_1\\ \gamma_2+i\zeta_2\end{smallmatrix}\end{bmatrix}$ with
$\gamma_1^2+\gamma_2^2+\zeta_1^2+\zeta_2^2=1$, then
\begin{align}\label{xyw}
x+iy=z=S_1b=\begin{bmatrix}z_1&z_2\end{bmatrix}\begin{bmatrix}\gamma_1+i\zeta_1\\ \gamma_2+i\zeta_2\end{bmatrix}, \quad
w=S_2b=\begin{bmatrix}w_1&w_2\end{bmatrix}\begin{bmatrix}\gamma_1+i\zeta_1\\ \gamma_2+i\zeta_2\end{bmatrix}.
\end{align}
Denoting
$\tilde{X}=\begin{bmatrix}\tilde{x}_1&\tilde{x}_2\end{bmatrix}$, $\tilde{Y}=\begin{bmatrix}\tilde{y}_1&\tilde{y}_2\end{bmatrix}$,
it can be easily verified that
\begin{align}
x=\begin{bmatrix}\tilde{X}&-\tilde{Y}\end{bmatrix}\begin{bmatrix}\gamma_1&\gamma_2&\zeta_1& \zeta_2\end{bmatrix}^{\top}, \qquad
y=\begin{bmatrix}\tilde{Y}&\tilde{X}\end{bmatrix}\begin{bmatrix}\gamma_1&\gamma_2& \zeta_1& \zeta_2\end{bmatrix}^{\top},
\end{align}
and
\begin{align}\label{eq-orthogonal}
x^{\top}y+y^{\top}x=\begin{bmatrix}\gamma_1& \gamma_2& \zeta_1& \zeta_2\end{bmatrix}
\begin{bmatrix}\tilde{X}^{\top}\tilde{Y}+\tilde{Y}^{\top}\tilde{X}&\tilde{X}^{\top}\tilde{X}-\tilde{Y}^{\top}\tilde{Y}\\
\tilde{X}^{\top}\tilde{X}-\tilde{Y}^{\top}\tilde{Y}&-(\tilde{X}^{\top}\tilde{Y}+\tilde{Y}^{\top}\tilde{X}) \end{bmatrix}
\begin{bmatrix}\gamma_1& \gamma_2& \zeta_1& \zeta_2\end{bmatrix}^\top,
\end{align}
\begin{align}\label{eq-equal-length}
x^{\top}x-y^{\top}y=\begin{bmatrix}\gamma_1& \gamma_2& \zeta_1& \zeta_2\end{bmatrix}
\begin{bmatrix}\tilde{X}^{\top}\tilde{X}-\tilde{Y}^{\top}\tilde{Y}&-(\tilde{X}^{\top}\tilde{Y}+\tilde{Y}^{\top}\tilde{X})\\
-(\tilde{X}^{\top}\tilde{Y}+\tilde{Y}^{\top}\tilde{X})&\tilde{Y}^{\top}\tilde{Y}-\tilde{X}^{\top}\tilde{X} \end{bmatrix}
\begin{bmatrix}\gamma_1& \gamma_2& \zeta_1& \zeta_2\end{bmatrix}^\top.
\end{align}
Obviously, the two matrices in \eqref{eq-orthogonal} and \eqref{eq-equal-length} are
symmetric Hamiltonian systems and they satisfy the property in Lemma \ref{Lemma3.2}. Hence we
can get the following lemma.

\begin{Lemma}\label{Lemma3.3}
Let $\phi_m, \phi_M$ be the two smallest singular values of
$\begin{bmatrix}\tilde{Y}&\tilde{X}\end{bmatrix}$ and
$\begin{bmatrix}\begin{smallmatrix}p_1\\q_1\end{smallmatrix}\end{bmatrix}, \begin{bmatrix}\begin{smallmatrix}p_2\\q_2\end{smallmatrix}\end{bmatrix}$
be the corresponding right singular vectors respectively. Define
\begin{align}\label{Omega}
\Omega=\begin{bmatrix}p_1&p_2&-q_1&-q_2\\q_1&q_2&p_1&p_2\end{bmatrix},
\end{align}
$\Phi=\diag (\phi_1, \phi_2, -\phi_1, -\phi_2)$ with $\phi_1=1-2\phi_m^2$, $\phi_2=1-2\phi_M^2$,
then
\begin{align}\label{specdecom}
\begin{bmatrix}\tilde{X}^{\top}\tilde{X}-\tilde{Y}^{\top}\tilde{Y}&-(\tilde{X}^{\top}\tilde{Y}+\tilde{Y}^{\top}\tilde{X})\\
-(\tilde{X}^{\top}\tilde{Y}+\tilde{Y}^{\top}\tilde{X})&\tilde{Y}^{\top}\tilde{Y}-\tilde{X}^{\top}\tilde{X} \end{bmatrix}=
\Omega\Phi \Omega^{\top},
\end{align}
and
\begin{align*}
\begin{bmatrix}\tilde{X}^{\top}\tilde{Y}+\tilde{Y}^{\top}\tilde{X}&\tilde{X}^{\top}\tilde{X}-\tilde{Y}^{\top}\tilde{Y}\\
\tilde{X}^{\top}\tilde{X}-\tilde{Y}^{\top}\tilde{Y}&-(\tilde{X}^{\top}\tilde{Y}+\tilde{Y}^{\top}\tilde{X}) \end{bmatrix}
=\Omega\left(\begin{array}{c|c}
\begin{array}{cc}
&\\
&\\
\end{array}
&\begin{array}{cc}
\phi_1&\\
&\phi_2\\
\end{array}\\ & \\[-2mm]
\hline
& \\[-2mm]
\begin{array}{cc}
\phi_1&\\
&\phi_2\\
\end{array}
&\begin{array}{cc}
&\\
&\\
\end{array}\\
\end{array}\right)\Omega^{\top}.
\end{align*}
\end{Lemma}

\begin{proof}
Since $(\tilde{X}^{\top}-i\tilde{Y}^{\top})(\tilde{X}+i\tilde{Y})=\begin{bmatrix}z_1&z_2\end{bmatrix}^{*}
\begin{bmatrix}z_1&z_2\end{bmatrix}=I_2$, so $\tilde{X}^{\top}\tilde{X}+\tilde{Y}^{\top}\tilde{Y}=I_2$ and
$\tilde{X}^{\top}\tilde{Y}=\tilde{Y}^{\top}\tilde{X}$.
Thus
\begin{align*}
\begin{bmatrix}\tilde{X}^{\top}\tilde{X}-\tilde{Y}^{\top}\tilde{Y}&-(\tilde{X}^{\top}\tilde{Y}+\tilde{Y}^{\top}\tilde{X})\\
-(\tilde{X}^{\top}\tilde{Y}+\tilde{Y}^{\top}\tilde{X})&\tilde{Y}^{\top}\tilde{Y}-\tilde{X}^{\top}\tilde{X}\end{bmatrix}
=\begin{bmatrix}I_2-2\tilde{Y}^{\top}\tilde{Y}&-2\tilde{Y}^{\top}\tilde{X}\\
-2\tilde{X}^{\top}\tilde{Y}&I_2-2\tilde{X}^{\top}\tilde{X} \end{bmatrix}
=I_4-2\begin{bmatrix}\tilde{Y}^{\top}\\ \tilde{X}^{\top}\end{bmatrix}\begin{bmatrix}\tilde{Y}&\tilde{X}\end{bmatrix}.
\end{align*}
From the above equation, it obviously holds that $\phi_1, \phi_2$ are the two nonnegative eigenvalues of
$\begin{bmatrix}\tilde{X}^{\top}\tilde{X}-\tilde{Y}^{\top}\tilde{Y}&-(\tilde{X}^{\top}\tilde{Y}+\tilde{Y}^{\top}\tilde{X})\\
-(\tilde{X}^{\top}\tilde{Y}+\tilde{Y}^{\top}\tilde{X})&\tilde{Y}^{\top}\tilde{Y}-\tilde{X}^{\top}\tilde{X}\end{bmatrix}$
with $\begin{bmatrix}p_1\\q_1\end{bmatrix}$, $\begin{bmatrix}p_2\\q_2\end{bmatrix}$ being the corresponding eigenvectors.
Note that $\begin{bmatrix}\tilde{X}^{\top}\tilde{X}-\tilde{Y}^{\top}\tilde{Y}&-(\tilde{X}^{\top}\tilde{Y}+\tilde{Y}^{\top}\tilde{X})\\
-(\tilde{X}^{\top}\tilde{Y}+\tilde{Y}^{\top}\tilde{X})&\tilde{Y}^{\top}\tilde{Y}-\tilde{X}^{\top}\tilde{X}\end{bmatrix}$
is a Hamiltonian matrix, thus  the results follow  immediately from Lemma \ref{Lemma3.1} and Lemma \ref{Lemma3.2}.
\end{proof}

Now by defining
\begin{align}\label{trans}
\begin{bmatrix}\mu_1& \mu_2& \nu_1& \nu_2\end{bmatrix}^{\top}=\Omega^{\top}
\begin{bmatrix}\gamma_1& \gamma_2& \zeta_1& \zeta_2\end{bmatrix}^{\top},
\end{align}
we have
\begin{align}\label{xdoty-xdotxplusydoty}
x^{\top}y+y^{\top}x=2\phi_1\mu_1\nu_1+2\phi_2\mu_2\nu_2, \qquad
x^{\top}x-y^{\top}y=\phi_1(\mu_1^2-\nu_1^2)+\phi_2(\mu_2^2-\nu_2^2).
\end{align}

\begin{Theorem}\label{Theorem3.2}
With the notations above, there exist $\mu_1, \mu_2, \nu_1, \nu_2\in\mathbb{R}$ such that
$x^{\top}y=0$ and $\|x\|_2=\|y\|_2=\frac{\sqrt{2}}{2}$. For these $\mu_1, \mu_2, \nu_1, \nu_2$,  let $\gamma_1, \gamma_2, \zeta_1, \zeta_2$ be computed from \eqref{trans}, where $\Omega$ is as in \eqref{Omega}. Then $x_{j+1}=x, x_{j+2}=y, v_{j+1}=\mbox{Re}(w)$ and $v_{j+2}=\mbox{Im}(w)$, where $w$ is computed by \eqref{xyw}, satisfy the constrains \eqref{eqcomplex-opt-2-constraina}-\eqref{eqcomplex-opt-2-constrainc},
and the value of the corresponding objective function in \eqref{eqcomplex-opt-2} will be no larger than $\frac{2(1-\sigma_2^2)}{\sigma_2^2}$.
\end{Theorem}
\begin{proof} It is easy to check that
all solutions of the following system of equations
\begin{align}\label{eqtwominimal}
\left\{\begin{array}{ll}
\phi_1\mu_1\nu_1+\phi_2\mu_2\nu_2&=0,\\
\phi_1(\mu_1^2-\nu_1^2)+\phi_2(\mu_2^2-\nu_2^2)&=0,\\
\mu_1^2+\mu_2^2+\nu_1^2+\nu_2^2&=1.
\end{array}\right.
\end{align}
are
\begin{equation}\label{munu1}
\begin{array}{lll}
\left\{\begin{split}
\mu_2&=\pm\sqrt{\frac{\phi_1}{\phi_1+\phi_2}-\nu_2^2}\\
\mu_1&=-\sqrt{\frac{\phi_2}{\phi_1}}\nu_2\\
\nu_1&=\pm\sqrt{\frac{\phi_2}{\phi_1+\phi_2}-\frac{\phi_2}{\phi_1}\nu_2^2}
\end{split}\right.& \textrm{and}&
\left\{\begin{split}
\mu_2&=\pm\sqrt{\frac{\phi_1}{\phi_1+\phi_2}-\nu_2^2}\\%[6pt]
\mu_1&=\sqrt{\frac{\phi_2}{\phi_1}}\nu_2\\%[6pt]
\nu_1&=\mp\sqrt{\frac{\phi_2}{\phi_1+\phi_2}-\frac{\phi_2}{\phi_1}\nu_2^2}
\end{split}\right.
\end{array}
\end{equation}
with $\nu_2^2\leq\frac{\phi_1}{\phi_1+\phi_2}$.
Note \eqref{xdoty-xdotxplusydoty} and $\|x\|_2^2+\|y\|_2^2=1$, so with the values in \eqref{munu1}, it holds that
$x^{\top}y=0$ and $\|x\|_2=\|y\|_2=\frac{\sqrt{2}}{2}$.
Since $\begin{bmatrix}z^\top&w^\top\end{bmatrix}^\top\in\mathcal{N}(M_{j+1})$, so
$\begin{bmatrix}\begin{smallmatrix}x_{j+1}&\ &x_{j+2}\\v_{j+1}&\ &v_{j+2}\end{smallmatrix}\end{bmatrix}=
\begin{bmatrix}\begin{smallmatrix}x&\ &y\\ \mbox{Re}(w)& \ &\mbox{Im}(w)\end{smallmatrix}\end{bmatrix}$ satisfy the constrains \eqref{eqcomplex-opt-2-constraina}-\eqref{eqcomplex-opt-2-constrainc} with
$\delta_1=\delta_2=\frac{\sqrt{2}}{2}$. Hence
\begin{align*}
&\|\delta_1v_{j+1}\|_2^2 + \|\delta_2v_{j+2}\|_2^2 + \beta_{j+1}^2 (\frac{\delta_1}{\delta_2} -\frac{\delta_2}{\delta_1})^2\\
=&2\|w\|_2^2=2(\gamma_1^2+\zeta_1^2)\frac{1-\sigma_1^2}{\sigma_1^2}+
2(\gamma_2^2+\zeta_2^2)\frac{1-\sigma_2^2}{\sigma_2^2}\leq\frac{2(1-\sigma_2^2)}{\sigma_2^2},
\end{align*}
which completes the proof of the theorem.
\end{proof}

From the proof of Theorem \ref{Theorem3.2} we can see that with such choice of $x_{j+1},x_{j+2},v_{j+1},v_{j+2}$, the value of the corresponding objective function is just $2\|w\|_2^2$. Define
$\xi_1=p_1^{\top}\Xi p_1, \xi_2=p_2^{\top}\Xi p_2, \eta_1=q_1^{\top}\Xi q_1, \eta_2=q_2^{\top}\Xi q_2,
\zeta_{12}=q_1^{\top}\Xi p_2, \zeta_{21}=q_2^{\top}\Xi p_1$,
with $\Xi=\diag\{ (1-\sigma_1^2)/\sigma_1^2, (1-\sigma_2^2)/\sigma_2^2\}$, it then follows
\begin{align}\label{lengthw}
\|w\|_2^2= \left\{ \begin{array}{ll}
\frac{\phi_2}{\phi_1+\phi_2}(\xi_1+\eta_1)+\frac{\phi_1}{\phi_1+\phi_2}(\xi_2+\eta_2)+
2\sqrt{\frac{\phi_2}{\phi_1}}\frac{\phi_1}{\phi_1+\phi_2}(\zeta_{21}-\zeta_{12}) & \textrm{if} \ \ (\mu_1\nu_2)\leq0,\\
\frac{\phi_2}{\phi_1+\phi_2}(\xi_1+\eta_1)+\frac{\phi_1}{\phi_1+\phi_2}(\xi_2+\eta_2)+
2\sqrt{\frac{\phi_2}{\phi_1}}\frac{\phi_1}{\phi_1+\phi_2}(\zeta_{12}-\zeta_{21}) & \textrm{if} \ \  (\mu_1\nu_2)>0.\\
\end{array} \right.
\end{align}
So in order to get a smaller $\|w\|_2$, we can take $\mu_1, \mu_2, \nu_1, \nu_2$ satisfying $\mu_1\nu_2\leq0$ if
$\zeta_{21}\leq\zeta_{12}$, and $\mu_1\nu_2>0$ if $\zeta_{21}>\zeta_{12}$.
%Moreover, if $x_{j+1}, x_{j+2}, v_{j+1}$
%and $v_{j+2}$ are taken in the way in Theorem \ref{Theorem3.2}, then $\|x_{j+1}\|_2=\|%x_{j+2}\|_2$ and the third
%term in \eqref{eqcomplex-opt-2} would be vanish. And in this case, we still only need to solve %a standard eigen-problem.

Till now we have proposed two strategies for computing $x_{j+1}, x_{j+2}, v_{j+1}, v_{j+2}$.
The first strategy computes $x_{j+1}, x_{j+2}, v_{j+1}$ and $v_{j+2}$ by using the Jacobi orthogonal process \eqref{Jocobi_x} and \eqref{Jocobi_v} with $z=u_1$ and $w=\frac{S_2Ve_1}{\sigma_1}$. While the  second one first computes $\mu_1, \mu_2, \nu_1, \nu_2$ by \eqref{munu1} satisfying $\mu_1\nu_2\leq0$ if
$\zeta_{21}\leq\zeta_{12}$, and $\mu_1\nu_2>0$ if $\zeta_{21}>\zeta_{12}$, and then compute $\gamma_1, \gamma_2, \zeta_1, \zeta_2$ from \eqref{trans}, where $\Omega$ is as in \eqref{Omega}, and finally set $x_{j+1}=x, x_{j+2}=y, v_{j+1}=\mbox{Re}(w)$ and $v_{j+2}=\mbox{Im}(w)$, where $x,y,w$ are computed by \eqref{xyw}. We cannot tell which strategy is better. So we suggest to apply both strategies, compare the corresponding values of the objective function and adopt the one which gives better results. Specifically, if the value of the objective function corresponding to the first strategy is smaller, we would update
$X_j$ and $T_j$ as
\begin{align}\label{updatecomplex1}
X_{j+2}=\begin{bmatrix}X_j &\delta_1x_{j+1}&\delta_2x_{j+2}\end{bmatrix}\in\mathbb{R}^{n\times (j+2)},\qquad
T_{j+2}=\begin{bmatrix}T_j &\delta_1v_{j+1}&\delta_2v_{j+2}\\
0 &\alpha_{j+1}&\delta\beta_{j+1}\\0&-\frac{1}{\delta}\beta_{j+1}&\alpha_{j+1}\end{bmatrix}\in\mathbb{R}^{(j+2)\times (j+2)},
\end{align}
where $\delta_1=\frac{1}{\|x_{j+1}\|_2}, \delta_2=\frac{1}{\|x_{j+2}\|_2}, \delta=\frac{\delta_2}{\delta_1}$. Otherwise, we update $X_j$ and $T_j$ as
\begin{align}\label{updatecomplex2}
X_{j+2}=\begin{bmatrix}X_j &\sqrt{2}x&\sqrt{2}y\end{bmatrix}\in\mathbb{R}^{n\times (j+2)},\quad
T_{j+2}=\begin{bmatrix}T_j &\sqrt{2}\mbox{Re}(w)&\sqrt{2}\mbox{Im}(w)\\0 &\alpha_{j+1}&\beta_{j+1}\\0&-\beta_{j+1}&\alpha_{j+1}\end{bmatrix}\in\mathbb{R}^{(j+2)\times(j+2)},
\end{align}
with $x, y$ and $w$ defined as in \eqref{xyw}. This completes the assignment of the complex
conjugate poles $\lambda_{j+1},\lambda_{j+2}=\bar{\lambda}_{j+1}$, and we can then continue with the next pole $\lambda_{j+3}$.
%So under certain conditions, we
%can obtain a satisfying value of the cost function in \eqref{eqcomplex-opt-2} just by solving a %standard eigen-problem,
%which makes our algorithm timesaving. And our experiment results in the upcoming section %will confirm that these suboptimal strategies in Theorem \ref{Theorem3.1} and Theorem %\ref{Theorem3.2} are satisfying.

%Essentially, we just take $z=S_1b\in\mathcal{R}(u_1)$ in Theorem \ref{Theorem3.1} and
%$z=S_1b\in\mathcal{R}(\begin{bmatrix}u_1&u_2\end{bmatrix})$ in Theorem \ref{Theorem3.2} %for $\begin{bmatrix}z^{\top}&w^{\top}\end{bmatrix}^{\top} \in\mathcal{N}(M_{j+1})$.
%If the condition in Theorem \ref{Theorem3.1} is not satisfied and $\|w\|_2$ is fairly large in %Theorem \ref{Theorem3.2},
These two strategies essentially choose $z$ from $\mathcal{R}(u_1)$ and $\mathcal{R}(\begin{bmatrix}u_1&u_2\end{bmatrix})$, respectively. If the results by these two strategies are not satisfactory, theoretically, we can choose $z$ from a higher dimensional space, i.e. $z\in \subspan\{u_1, u_2, \ldots, u_k\}, k\geq3$, with $u_l$ being the $l$-th column of $U$.
However the resulted optimization problem is much more complicated. More importantly, numerical examples show that these two strategies with $k=1,2$ can produce fairly satisfying results for most problems.

\subsection{Algorithm}
In this part, we give the framework of our algorithm.
\begin{algorithm}
\renewcommand{\algorithmicrequire}{\textbf{Input:}}
\renewcommand\algorithmicensure {\textbf{Output:} }
\caption{ Framework of our {\bf Schur-rob} algorithm.}
\label{alg:Framwork}
\begin{algorithmic}[1]
\REQUIRE ~~\\
$A, B$ and $\mathfrak{L}=\{\lambda_1,\dots,\lambda_n\}$ (complex conjugate poles appear in pairs).
\ENSURE ~~\\
The feedback matrix $F$.
\STATE If $\lambda_1$ is real, compute $x_1$ by \eqref{x1real} and set $X_1=x_1,T_1=\lambda_1,j=1$.
If $\lambda_1$ is non-real, compute $x_1,x_2$ by  \eqref{x1x2get}, \eqref{gammamunu}, \eqref{munuinitial},  and set $X_2,T_2$ as in
\eqref{initialnonreal}, $j=2$. \label{code:fram:extract}
\WHILE{$j<n$}
\IF{ $\lambda_{j+1}$ is real}
\STATE Find $S=\begin{bmatrix}S_1^{\top}&S_2^{\top}\end{bmatrix}^{\top}$, whose columns form an orthonormal basis of $\mathcal{N}(M_{j+1})$ in \eqref{M};
\STATE Compute $y$ by \eqref{eqreal-opt-equal-3};
\STATE Compute $x_{j+1}$ and $v_{j+1}$ by \eqref{eqrealxv}, update $X_j$ and $T_j$ as
\eqref{updatereal} and set $j=j+1$.
\ELSE
   \STATE Find $S=\begin{bmatrix}S_1^{\top}&S_2^{\top}\end{bmatrix}^{\top}$, whose columns form an orthonormal basis of $\mathcal{N}(M_{j+1})$ in \eqref{Mcomplex};
   \STATE Compute the SVD of $S_1$ as $S_1=U\Sigma V^{*}$;
   \IF{ $\mbox{Re}(Ue_1)$ and $\mbox{Im}(Ue_1)$ are linearly independent}
      \STATE Compute $x_{j+1}, x_{j+2}, v_{j+1}, v_{j+2}$ by \eqref{Jocobi_x} and \eqref{Jocobi_v} with $z=\frac{S_1Ve_1}{\sigma_1}, w=\frac{S_2Ve_1}{\sigma_1}$;
      \STATE Set $\delta_1=\frac{1}{\|x_{j+1}\|_2}, \delta_2=\frac{1}{\|x_{j+2}\|_2}$ and $\delta=\frac{\delta_2}{\delta_1}$;
      \STATE Compute $dep_1=\|\delta_1v_{j+1}\|_2^2+\|\delta_2v_{j+2}\|_2^2+\beta_{j+1}^2(\delta-\frac{1}{\delta})^2$;
   \ELSE
   \STATE  Set $dep_1=\infty$;
   \ENDIF
\STATE Let $\tilde{X}=\begin{bmatrix}\tilde{x}_1&\tilde{x}_2\end{bmatrix}$, $\tilde{Y}=\begin{bmatrix}\tilde{y}_1&\tilde{y}_2\end{bmatrix}$  with $\tilde{x}_1, \tilde{y}_1, \tilde{x}_2, \tilde{y}_2$ defined as in \eqref{tildexy}, and compute the spectral decomposition \eqref{specdecom};
\STATE Compute $\mu_1, \mu_2, \nu_1, \nu_2$ by \eqref{munu1} satisfying $\mu_1\nu_2\leq0$ if
$\zeta_{21}\leq\zeta_{12}$, and $\mu_1\nu_2>0$ if $\zeta_{21}>\zeta_{12}$, and then compute $\gamma_1, \gamma_2, \zeta_1, \zeta_2$ from \eqref{trans}, where $\Omega$ is as in \eqref{Omega};
%\STATE Set $x_{j+1}=x, x_{j+2}=y, v_{j+1}=\mbox{Re}(w)$ and $v_{j+2}=\mbox{Im}(w)$, where $x,y,w$ are computed by \eqref{xyw} and compute
%       $dep_2=2[(\gamma_1^2+\zeta_1^2)\frac{1-\sigma_1^2}{\sigma_1^2}+(\gamma_2^2+\zeta_2^2)\frac{1-\sigma_2^2}{\sigma_2^2}]$;
\STATE Compute $z$, $w$ by \eqref{xyw}, set $x_{j+1}=\mbox{Re}(z), x_{j+2}=\mbox{Im}(z), v_{j+1}=\mbox{Re}(w)$ and $v_{j+2}=\mbox{Im}(w)$.
 Compute
$dep_2=2[(\gamma_1^2+\zeta_1^2)\frac{1-\sigma_1^2}{\sigma_1^2}+(\gamma_2^2+\zeta_2^2)\frac{1-\sigma_2^2}{\sigma_2^2}]$;
\STATE If $dep_1<dep_2$, update $X_{j}$ and $T_{j}$ as in \eqref{updatecomplex1}; otherwise, update them as in \eqref{updatecomplex2}. Set $j=j+2$.
\ENDIF
\ENDWHILE
\STATE Set $X=X_n, T=T_n$, and compute $F$ by \eqref{eqsolveoff}.
\end{algorithmic}
\end{algorithm}

%\begin{Remark}
%With $\lambda_1\in\mathbb{R}$ fixed, the ad hoc choice of $x_1$ in \textbf{Schur-rob} does not take the
%advantage of the freedom of $x_1$, and this may influence the final results, including the robustness.
%Inspired by \textbf{O-SCHUR} \cite{Chu2}, we attempt to obtain a better
%feedback control by selecting some other $x_1$ distinct from that in \textbf{Schur-rob}.
%Concretely, with many different $x_1$ being picked up randomly, we perform \textbf{Schur-rob}
%from the statement $2$ to the end to get several $F$, then the solution to the {\bf SFRPA} is assigned to be the
%one corresponding to the minimum departure from normality. And we denote such
%optimal method relating to $x_1$  as ``\textbf{O-Schur-rob}".
%\end{Remark}

\section{Numerical Examples}

In this section, we give some numerical examples to illustrate the performance of our \textbf{Schur-rob} algorithm,
and compare it with some of the different versions of \textbf{SCHUR} in \cite{Chu2}, the MATLAB functions \textbf{robpole} \cite{Tits} and \textbf{place} \cite{KNV}.
Each algorithm computes a feedback matrix $F$ such that the eigenvalues of $A+BF$ are those given in $\mathfrak{L}$, and $A+BF$ is robust.
When applying \textbf{robpole} to all test examples, we set the maximum number of sweep to be the default value $5$.
All calculations are carried out on  an Intel\textregistered    Core\texttrademark i3, dual core, 2.27 GHz machine, with $2.00$ GB RAM.
MATLAB R2012a is used with machine epsilon $\epsilon\approx 2.2\times 10^{-16}$.

With $\lambda_1\in\mathbb{R}$ fixed, the choice of $x_1$ in \textbf{Schur-rob} ignores the freedom of $x_1$. Inspired by \textbf{O-SCHUR} \cite{Chu2}, we may regard $x_1$ as a free parameter and manage to optimize the robustness. Specifically, we may run \textbf{Schur-rob} with several different choices of $x_1$, and keep the solution $F$  corresponding to the minimum departure from normality. We denote such method  as ``\textbf{O-Schur-rob}".

In this section, results on precision and robustness obtained by different algorithms are displayed. Here the precision refers to the accuracy of the eigenvalues of computed $A_c=A+BF$, compared with the prescribed poles in $\mathfrak{L}$. Precisely, we list
\[
precs=\left\lfloor\min_{1\le j\le n}(-\log(|\frac{\lambda_j-\hat{\lambda}_j}{\lambda_j}|))
\right\rfloor,
\]
where $\hat{\lambda}_j, j=1, \ldots, n$ are eigenvalues of computed  $A_c=A+BF$.
Larger values of $precs$ indicate more accurate computed eigenvalues.
The robustness is, however, more complicated, since different measures of
robustness are used in these algorithms.
Specifically, let the spectral decomposition and the real Schur decomposition of $A+BF$ respectively be
\[
A+BF=X\Lambda X^{-1},\qquad A+BF=UTU^\top,
\]
where $\Lambda$ is a diagonal matrix whose diagonal elements are those in $\mathfrak{L}$, $U$ is orthogonal, and $T$ is the real Schur form. The MATLAB function \textbf{place} tends to minimize $\|X^{-1}\|_F$  and \textbf{robpole} aims to maximum $|\det(X)|$.
Both measures are closely related to the condition number $\kappa_F(X)=\|X\|_F\|X^{-1}\|_F$.
While different versions of \textbf{SCHUR} \cite{Chu2}
and our \textbf{Schur-rob} try to minimize the departure from normality of $A_c=A+BF$. Hence, in the following tests,
we adopt the following two  measures of robustness:  the departure from normality of $A_c$ (denoted as ``dep.") and the
condition number of $X$ (denoted as $``\kappa_F(X)"$).%, or either one for the currently applied algorithms.
%We also list the Frobenius norm of the feedback matrix $F$ (denoted as $`` \|F\|_F"$),
%which is also regarded as a  robustness measurement in some literature.

\begin{Example}{\rm~  Let
\begin{align*}
\begin{array}{ll}
A=\begin{bmatrix}1&0&0\\ 0&I_{n-2}&0\\0& 0.5\times e^\top& 0.5 \end{bmatrix}, &
B=\begin{bmatrix}I_{n-1}\\0\end{bmatrix}, \\
\mathfrak{L}=\{ randn(1, n-2), \  0.5+ki,\  0.5-ki\},
\end{array}
\end{align*}
where $e^\top$ is the row vector with its all entries being $1$, ``$randn(1, n-2)$" is a row vector of dimension $n-2$,
generated by the MATLAB function \textbf{randn}. We set $k$ as $1e+1, 1e+2, 1e+3, 1e+4, 1e+5$, and apply the four algorithms \textbf{SCHUR},  \textbf{SCHUR-D}, \textbf{O-SCHUR} and \textbf{Schur-rob} on these examples,
where ``\textbf{SCHUR-D}" denotes the algorithm combining the $D_k$ varying strategy in \cite{Chu2} with
\textbf{SCHUR}. In \cite{Chu2}, the author points out that
minimizing the departure from normality via the $D_k$ varying
technique can be achieved by optimizing the condition
number of $X^\top X$ or $X$, which actually is hard to realize. So here,
the numerical results associated with ``\textbf{SCHUR-D}" are obtained by taking
many different vectors from the null space of $(6)$ in \cite{Chu2}, which lead to orthogonal columns in $X$ when placing
complex conjugate poles, and adopting the
one owning the minimal departure from normality as the solution to the {\bf SFRPA}.
All numerical results are summarized in Table~\ref{table_add}, which shows that our algorithm
outperforms \textbf{SCHUR} and \textbf{O-SCHUR}  on these examples with complex conjugate poles to be assigned.

\tabcolsep 0.0001in
\begin{table}
\centering
\begin{tabular}{c|c|c|c|c|c|c|c|c}
\Xhline{2pt}
\multirow{2}*{$(n,k)$} & \multicolumn{4}{c|}{$dep.$}&\multicolumn{4}{c}{$precs$} \\
\Xcline{2-9}{0.1pt}
& \textbf{SCHUR} & \textbf{SCHUR-D}& \textbf{O-SCHUR}&\textbf{Schur-rob} & \textbf{SCHUR} & \textbf{SCHUR-D}&\textbf{O-SCHUR}&\textbf{Schur-rob} \\
\Xhline{1pt}
(4, 1e+1)&9.5e+1&  2.2e+1 &4.3e+1 & 2.7e+0 & 14&14 & 14&15\\
(4, 1e+2)&1.5e+4 &8.2e+2 & 1.4e+4&  3.3e+2 & 11& 13&11 & 14\\
(4, 1e+3)&1.4e+6& 6.6e+4 &1.2e+6 & 6.6e+2&7&8 &7 & 10\\
(4, 1e+4)&2.9e+8&9.9e+5 &4.3e+7&1.0e+4 &4& 10&6 &13 \\
(4, 1e+5)&1.8e+10&7.3e+6 &1.2e+10 &3.8e+5 & 3& 7 &3 & 10\\
\Xhline{1pt}
(20,1e+1)&4.0e+1& 7.6e+0&1.7e+1 &4.6e+0 & 13&14&14 &14 \\
(20,1e+2)& 7.7e+4& 2.6e+2&2.4e+2&1.8e+1 & 9 &12&11 &12 \\
(20,1e+3)&2.0e+5&4.4e+3& 9.3e+4&4.7e+2 &9& 11&10&12 \\
(20,1e+4)&3.2e+7& 2.4e+4&5.2e+6 &1.9e+3 & 6&10& 8& 11  \\
(20,1e+5)&1.7e+9& 1.2e+6&8.8e+8 &6.0e+4 & 3&9&6 &10\\
\Xhline{1pt}
(50,1e+1)& 1.1e+1& 2.9e+0&4.4e+0 &4.4e+0 & 13& 12 &13 &13  \\
(50,1e+2)& 2.0e+4& 5.9e+2&8.8e+2 &1.8e+1 & 10& 12 & 11 & 12 \\
(50,1e+3)& 1.1e+6& 7.8e+2&5.8e+4 &5.5e+2 & 8 &11& 9 &12 \\
(50,1e+4)& 8.8e+7&3.2e+4&9.6e+6 &2.1e+3 & 6 &10&7 & 11 \\
(50,1e+5)& 8.4e+9& 2.0e+5&4.8e+8 &3.7e+4 & 3& 9 & 5&10 \\
\Xhline{2pt}
\end{tabular}
\caption{Numerical results for Example 4.1}
\label{table_add}
\end{table}
}\end{Example}

We now compare our \textbf{Schur-rob},  \textbf{O-Schur-rob} algorithms
with the MATLAB functions \textbf{place}, \textbf{robpole} and
the \textbf{SCHUR}, \textbf{O-SCHUR} algorithms by applying them on some benchmark sets.
The tested benchmark sets include eleven illustrated examples from \cite{BN}, ten multi-input CARE examples and nine
multi-input DARE examples in benchmark collections \cite{AB, AB2}.
All examples are numbered in the order as they appear in the references.
%In addition, we also present the numerical results produced by the \textbf{O-Schur-rob}.

\begin{Example}{\rm~
The first benchmark set includes eleven small examples from \cite{BN}. Applying the six algorithms on these examples,
all algorithms produce comparable precisions of the assigned poles, which are greater than $10$, and we omit the results here.
Table~\ref{table_1} lists two measures of robustness, i.e. $dep.$ and $\kappa_F(X)$, for five examples.
The results are generally comparable.
The remaining six examples are not displayed in the table,
as the results of the six algorithms applying on these examples are quite similar.

\tabcolsep 0.02in
\begin{table}
\centering
\begin{tabular}{c|c|c|c|c|c|c}
\Xhline{2pt}
\multicolumn{2}{c|}{$num.$} & 5 & 7& 8& 9& 11\\
\Xcline{1-7}{1pt}
\multirow{6}*{$dep.$}&\textbf{place}& 7.4e-1 &3.5e+0 &1.3e+1 &1.2e+1 & 2.5e-3 \\
 & \textbf{robpole} &7.4e-1  & 3.4e+0&5.0e+0 &1.2e+1 & 3.6e-1 \\
& \textbf{SCHUR} & 7.2e-1&7.2e+0 &7.0e+0 &1.9e+1 & 2.3e+0  \\
& \textbf{O-SCHUR}  & 7.1e-1& 4.8e+0& 6.0e+0&1.7e+1 & 6.0e-1  \\
& \textbf{Schur-rob} &7.2e-1 &3.7e+0 &7.5e+0 &1.8e+1 & 2.4e-1\\
& \textbf{O-Schur-rob} &7.1e-1 &3.2e+0 &3.3e+0 &1.1e+1 & 1.4e-1\\
\Xcline{1-7}{1pt}
\multirow{6}*{$\kappa_F(X)$}
&\textbf{place}& 1.5e+2& 1.2e+1&3.7e+1& 2.4e+1&4.0e+0 \\
 & \textbf{robpole}&1.5e+2&1.2e+1& 6.2e+0&2.4e+1&4.1e+0\\
& \textbf{SCHUR} &2.7e+3&1.3e+2& 1.1e+1&5.6e+1&6.0e+0 \\
& \textbf{O-SCHUR} & 1.1e+3&4.5e+1&7.5e+0&5.5e+1&4.1e+0\\
& \textbf{Schur-rob} &1.9e+3&2.5e+1&1.2e+1&5.8e+1&4.1e+0\\
& \textbf{O-Schur-rob} &1.2e+3&2.2e+1&9.6e+0&3.3e+1&4.0e+0\\
\Xcline{1-7}{1pt}
\Xhline{2pt}
\end{tabular}
\caption{Robustness of the closed-loop system for the examples from \cite{BN}}
\label{table_1}
\end{table}

Now we apply the six algorithms on ten CARE and nine DARE examples from the
SLICOT CARE/DARE benchmark collections \cite{AB, AB2}.
Table~\ref{table_2} to Table~\ref{table_5} present the numerical results, respectively.
The ``-"s  in the first columns in Table~\ref{table_3} and Table~\ref{table_5}
corresponding to \textbf{place}, \textbf{robpole}, \textbf{SCHUR} and \textbf{O-SCHUR}
mean that all four algorithms fail to output a solution, since the multiplicity of some pole is greater than $m$.
Note that the $``precs"$ in the last six columns associated with \textbf{SCHUR} and \textbf{O-SCHUR} in Table~\ref{table_2}
and those in the third and eighth columns in Table~\ref{table_3} are also `` -"s, which suggest that there exists at least one
eigenvalue of $A+BF$, which owns no relative accuracy compared with the assigned poles.
From Table~\ref{table_2}, we know that the relative accuracy $``precs"$ of the poles in example $4$ and $5$
corresponding to \textbf{Schur-rob} and \textbf{O-Schur-rob} are lower than those produced by \textbf{place} and \textbf{robpole}.
And the reason is that there are semi-simple eigenvalues in both examples.
So how to dispose the issue that semi-simple eigenvalues can achieve higher relative accuracy deserves further exploration
and we will treat it in a separate paper.
For the sixth column in Table~\ref{table_2}, $``precs"$ from our algorithms are also smaller than those obtained
from \textbf{place} and \textbf{robpole} for the existence of poles which are relatively badly
separated from the imaginary axis.
And this is a weakness of our algorithm.

\begin{table}
\begin{minipage}{0.5\textwidth}
    \centering
    \begin{tabular}{ccccccccccc}
\Xhline{2pt}
\multicolumn{10}{c}{ $precs$}\\
\Xhline{0.4pt}
& $1$& $2$&$3$ & $4$& $5$ & $6$& $7$&  $8$& $9$&$10$\\
\textbf{place}&14&14&11&11&11&9	&14&11&	13&11\\
\textbf{robpole}&14&14&12&13&12&11&14&14&13&10\\
\textbf{SCHUR}&12&13&9&6&-&-&-&-&-&-\\
\textbf{O-SCHUR}&14&16&10&7&-&-&-&-&-&-\\
\textbf{Schur-rob}&14&14&12&8&9&6&14&14&12&9\\
\textbf{O-Schur-rob}&15&15&13&8&9&6&14&14&12&9\\
\Xhline{2pt}
\end{tabular}
\caption{Accuracy for CARE examples}
\label{table_2}
  \end{minipage}
  \begin{minipage}{0.5\textwidth}
    \centering
    \begin{tabular}{cccccccccc}
\Xhline{2pt}
\multicolumn{10}{c}{ $precs$}\\
\Xhline{0.4pt}
& $1$& $2$&$3$ & $4$& $5$ & $6$& $7$&  $8$& $9$\\
\textbf{place}&-&15&14&	14&7&11	&5&	-&	13\\
\textbf{robpole}&-&15&14&14&7&11&1&-&13\\
\textbf{SCHUR}&-&1&-&14&7&8&1&-&12\\
\textbf{O-SCHUR}&-&1&-&14&8&9&2&-&15\\
\textbf{Schur-rob}&15&15&15&15&8&10&4&-&12\\
\textbf{O-Schur-rob}&15&15&15&15&8&10&4&-&13\\
\Xhline{2pt}
\end{tabular}
    \caption{Accuracy for DARE examples}
    \label{table_3}
  \end{minipage}%
\end{table}

\tabcolsep 0.001in
\begin{table}
\centering
\begin{tabular}{c|c|c|c|c|c|c|c|c|c|c|c}
\Xhline{2pt}
\multicolumn{2}{c|}{$num.$} & 1& 2& 3& 4& 5& 6 & 7& 8& 9& 10\\
\Xcline{1-12}{1pt}
\multirow{5}*{$dep.$}
&\textbf{place}&5.2e+0&3.0e-1&7.3e+2&1.5e+6&2.9e+6&2.3e+7&7.6e+0&2.2e+1&6.1e+0&4.9e+9  \\
 & \textbf{robpole}& 5.2e+0&2.9e-1&5.7e+2& 7.5e+5&2.9e+6&2.3e+7&8.1e+0&2.0e+1& 6.0e+0&3.8e+9 \\
& \textbf{SCHUR} &8.4e+1& 7.2e+0&5.0e+2&1.7e+6&3.0e+9 &5.3e+7&6.2e+1& 8.9e+2&7.5e+0& 4.4e+17\\
& \textbf{O-SCHUR} &4.7e+1&2.6e+0& 3.8e+2& 8.0e+5&5.4e+8 &2.6e+7  &7.3e+0&1.7e+2&6.8e+0& 2.3e+17  \\
& \textbf{Schur-rob}&7.6e+0&3.0e-1& 1.4e+2&1.1e+5&7.3e+6&2.3e+7&7.5e+0&2.1e+1&8.4e+0&2.2e+10 \\
& \textbf{O-Schur-rob}&7.3e+0&2.6e-1& 1.4e+2&1.1e+5&2.5e+6&2.3e+7&6.8e+0&2.0e+1&6.8e+0&2.2e+10 \\
\Xcline{1-12}{1pt}
\multirow{5}*{$\kappa_F(X)$}
&\textbf{place}&7.4e+0&8.0e+0&4.3e+1&1.7e+15&8.5e+4&4.8e+6&1.6e+1&9.8e+1&1.5e+2&2.3e+6 \\
 & \textbf{robpole} &7.3e+0&8.0e+0&4.2e+1&2.2e+7&8.9e+4&3.2e+6&1.6e+1&9.0e+1&1.4e+2& 2.3e+6\\
& \textbf{SCHUR} &2.2e+2&1.0e+1&1.7e+3&9.1e+9& 6.0e+11&4.0e+13 &3.5e+8& 6.1e+9&1.3e+9& 4.6e+13\\
& \textbf{O-SCHUR} &1.2e+2& 5.1e+1&2.1e+3&1.0e+9&2.4e+10 &1.2e+8&1.0e+8& 3.7e+9& 4.1e+9&5.7e+13\\
& \textbf{Schur-rob} &1.1e+1& 8.2e+0&9.2e+2&9.0e+7&2.0e+6&3.2e+8&3.3e+1&5.7e+2&6.5e+3&4.3e+6\\
& \textbf{O-Schur-rob} &1.0e+1&8.0e+0&9.1e+2&6.5e+7&1.3e+6&1.2e+8&2.8e+1&4.2e+2&3.4e+3&4.3e+6\\
\Xcline{1-12}{1pt}
\Xhline{2pt}
\end{tabular}
\caption{Robustness of the closed-loop system matrix for ten CARE examples}
\label{table_4}
\end{table}

\tabcolsep 0.02in
\begin{table}
\centering
\begin{tabular}{c|c|c|c|c|c|c|c|c|c|c}
\Xhline{2pt}
\multicolumn{2}{c|}{$num.$} & 1& 2& 3& 4& 5& 6 & 7& 8& 9\\
\Xcline{1-11}{1pt}
\multirow{5}*{$dep.$}
&\textbf{place}&-&2.2e-1&3.9e-1&4.3e-1&1.7e+0&1.4e+0&2.3e+1&4.3e+7&8.9e+0 \\
 & \textbf{robpole}&-&2.2e-1&3.9e-1&3.6e-1&1.7e+0&1.3e+0&1.8e+1&3.9e+12& 8.0e+0\\
& \textbf{SCHUR} &-&4.1e-1&1.1e+2&5.9e-1&1.8e+0&1.1e+1&3.2e+2&3.4e+2&1.1e+1\\
& \textbf{O-SCHUR} &-&3.3e-1&4.9e+1&4.1e-1&1.7e+0&1.1e+0& 1.7e+2&1.2e+1&8.0e+0 \\
& \textbf{Schur-rob}&1.0e-1&2.5e-1&1.3e+0&3.4e-1& 1.7e+0&2.0e+0&1.9e+1&9.8e+0&9.9e+0\\
& \textbf{O-Schur-rob}&1.0e-1&2.5e-1&1.3e+0&3.4e-1& 1.7e+0&1.2e+0&1.8e+1&9.4e+0&6.6e+0\\
\Xcline{1-11}{1pt}
\multirow{5}*{$\kappa_F(X)$}
&\textbf{place}&-&5.2e+0& 4.9e+0&5.4e+0&1.8e+1&1.3e+1&2.3e+8&9.2e+292&3.4e+2\\
& \textbf{robpole} &-&5.2e+0&5.0e+0&5.3e+0&1.8e+1&1.2e+1&2.9e+8&1.3e+308&3.0e+2\\
& \textbf{SCHUR} &-&4.0e+7&1.2e+9&5.7e+0& 1.8e+1&5.8e+3&1.9e+11&2.8e+295&4.7e+3\\
& \textbf{O-SCHUR} &-&3.3e+7&8.0e+8&5.4e+0&1.8e+1&1.7e+3&2.0e+11&3.3e+295&2.6e+3\\
& \textbf{Schur-rob} &7.1e+15&5.5e+0&5.6e+0&7.2e+0&1.8e+1&3.8e+1&1.7e+9&5.6e+292&2.2e+4\\
& \textbf{O-Schur-rob} &2.5e+15&5.5e+0&5.5e+0&7.2e+0&1.8e+1&3.8e+1&1.2e+9&5.6e+292&4.7e+3\\
\Xcline{1-11}{1pt}
\Xhline{2pt}
\end{tabular}
\caption{Robustness of the closed-loop system matrix for nine DARE examples}
\label{table_5}
\end{table}
}\end{Example}

We now test the five methods \textbf{place}, \textbf{robpole},
\textbf{SCHUR}, \textbf{O-SCHUR} and \textbf{Schur-rob} on some random examples generated by the MATLAB function \textbf{randn}.

\begin{Example}{\rm~
This test set includes $33$ examples where
$n$ varies from 3 to 25 increased by 2, and $m$ is set to be $2, \lfloor\frac{n}{2}\rfloor, n-1$ for each $n$. The examples are generated as following. We first randomly generate the matrices $A, B$ and $F$ by the MATLAB function \textbf{randn}, and then get $\mathfrak{L}$ using the
MATLAB function \textbf{eig}, that is, $\mathfrak{L}=eig(A+BF)$. We then apply the five algorithms
on the $A,B$ and $\mathfrak{L}$ as input.
%\textbf{place}, \textbf{robpole},\textbf{SCHUR}, \textbf{O-SCHUR} and \textbf{Schur-rob}

Fig.~\ref{fig1} to Fig.~\ref{fig4},  respectively exhibit the departure from normality of the computed $A_c$,
the condition number of the eigenvector matrix $X$,
the relative accuracy of the poles and the CPU time of the five algorithms
applied on these randomly generated examples.
In these figures, the $x$-axis represents the number of the $33$ different $(n,m)$. For example, $(3,2)$, $(5,2)$ and $(5,4)$ correspond to $1$, $2$ and $3$ in the $x$-axis, respectively.
And the values along the $y$-axis are the mean values over 50 trials for a certain $(n,m)$.

\begin{figure}[H]
\centering
\begin{tabular}{ccc}
    \begin{minipage}[t]{0.5\textwidth}
    \includegraphics[width=3.in,height=2.1in]{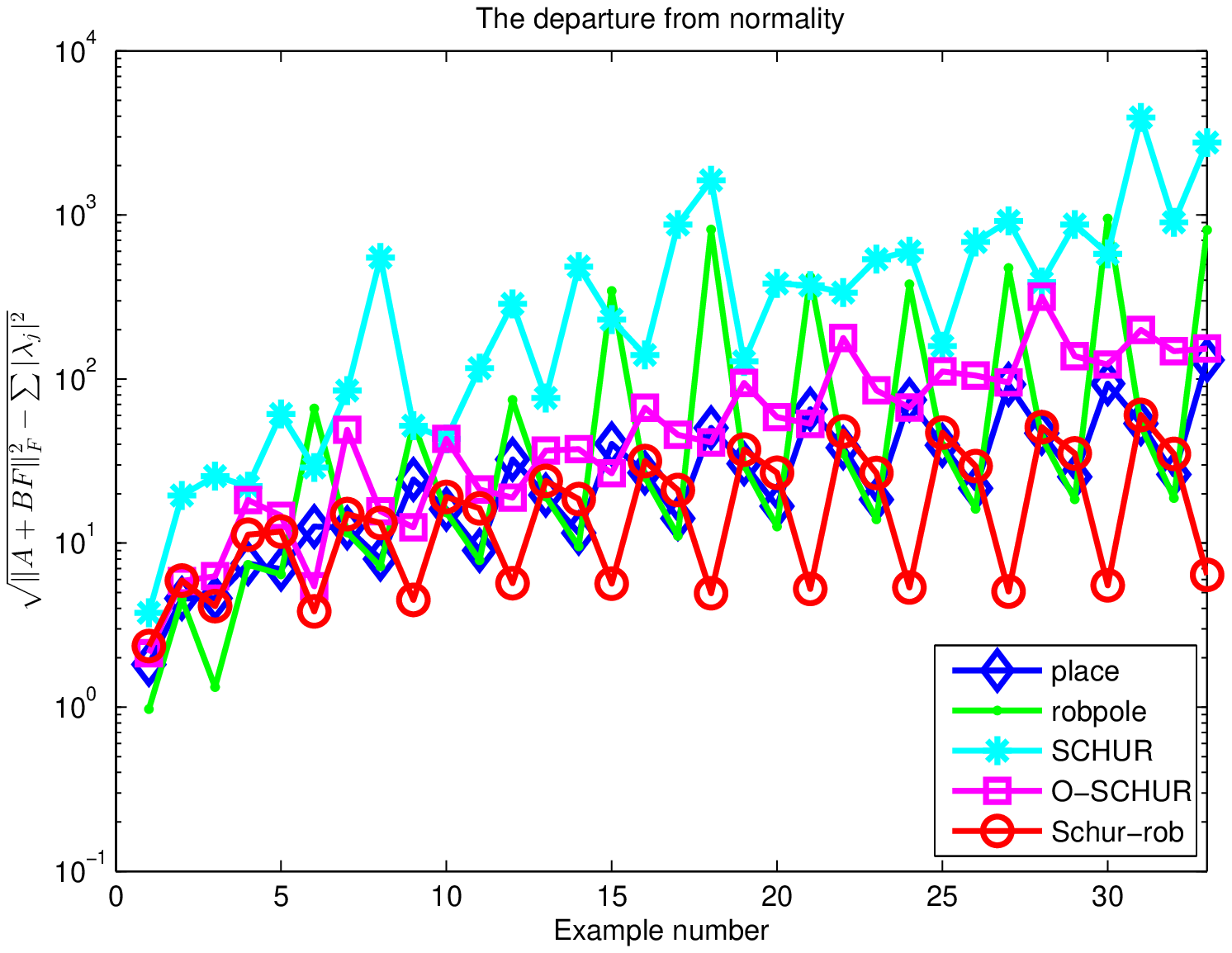}
    \caption{$dep.$ over 50 trials}\label{fig1}
    \end{minipage}&
    \begin{minipage}[t]{0.5\textwidth}
    \includegraphics[width=3.in,height=2.1in]{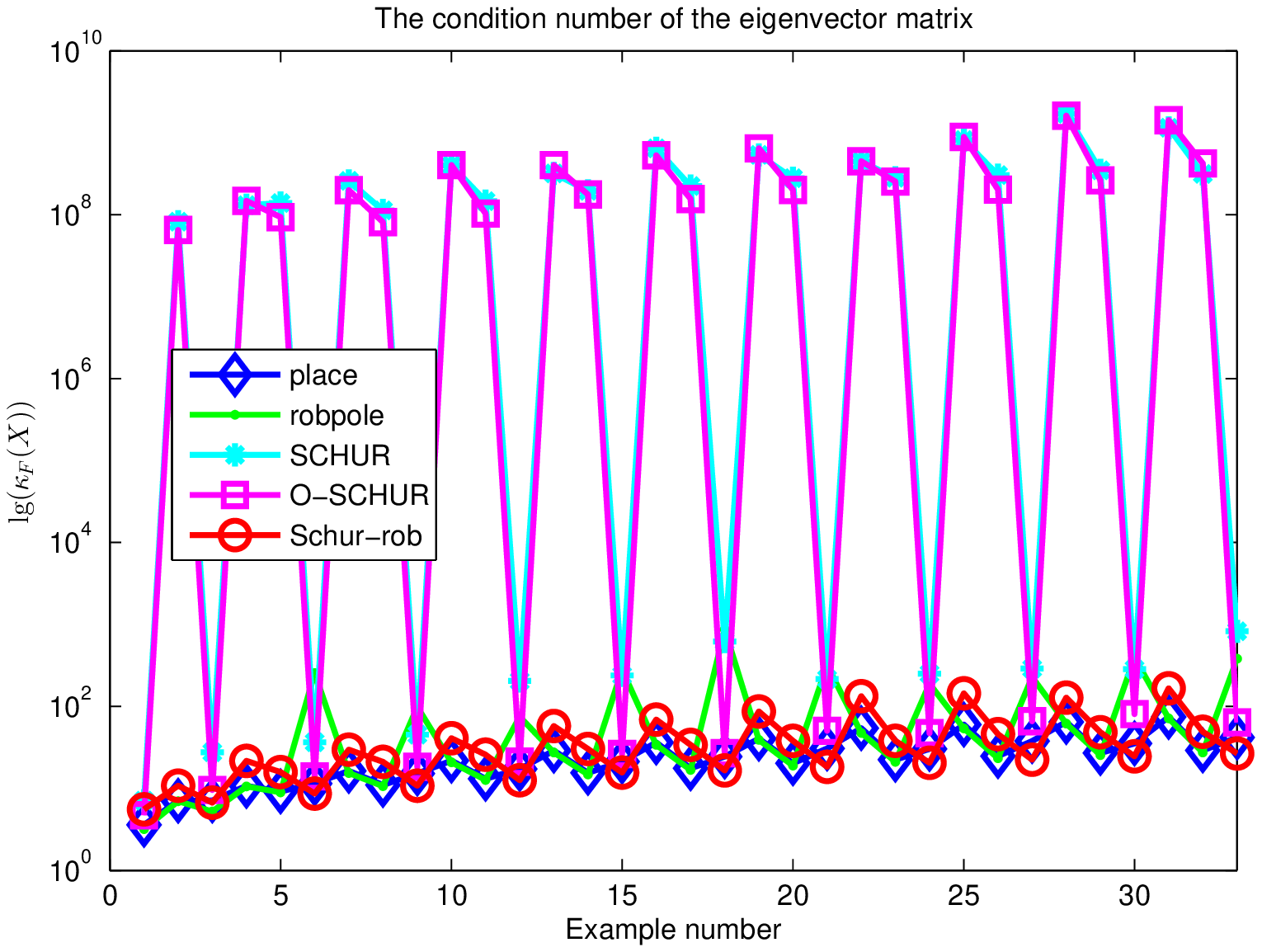}
    \caption{$\kappa_F(X)$ over 50 trials}\label{fig2}
    \end{minipage}
\end{tabular}
\end{figure}

\begin{figure}[H]
\centering
\begin{tabular}{ccc}
    \begin{minipage}[t]{0.5\textwidth}
    \includegraphics[width=3.in,height=2.1in]{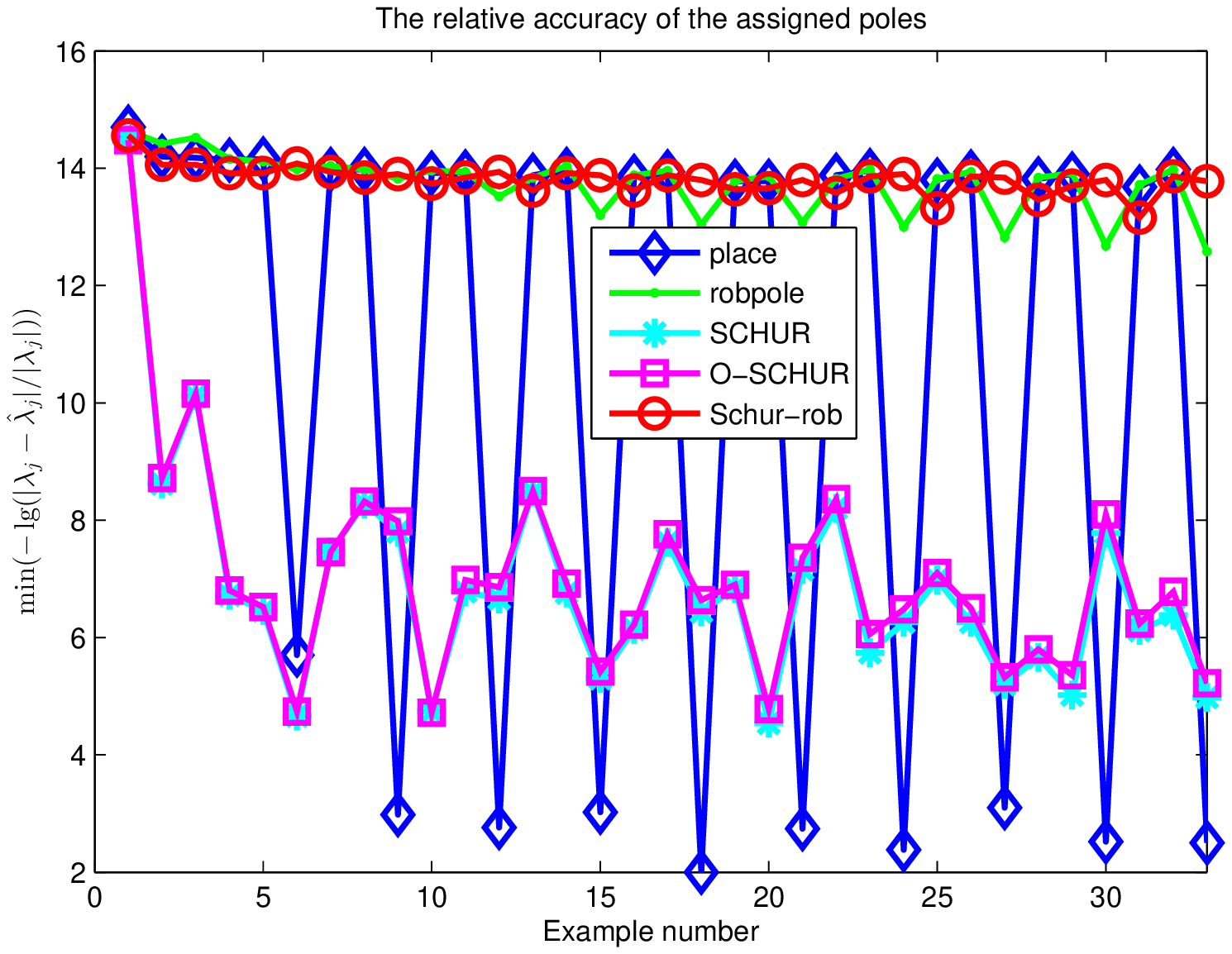}
    \caption{$precs$ over 50 trials}\label{fig3}
    \end{minipage}&
    \begin{minipage}[t]{0.5\textwidth}
   \includegraphics[width=3.in,height=2.1in]{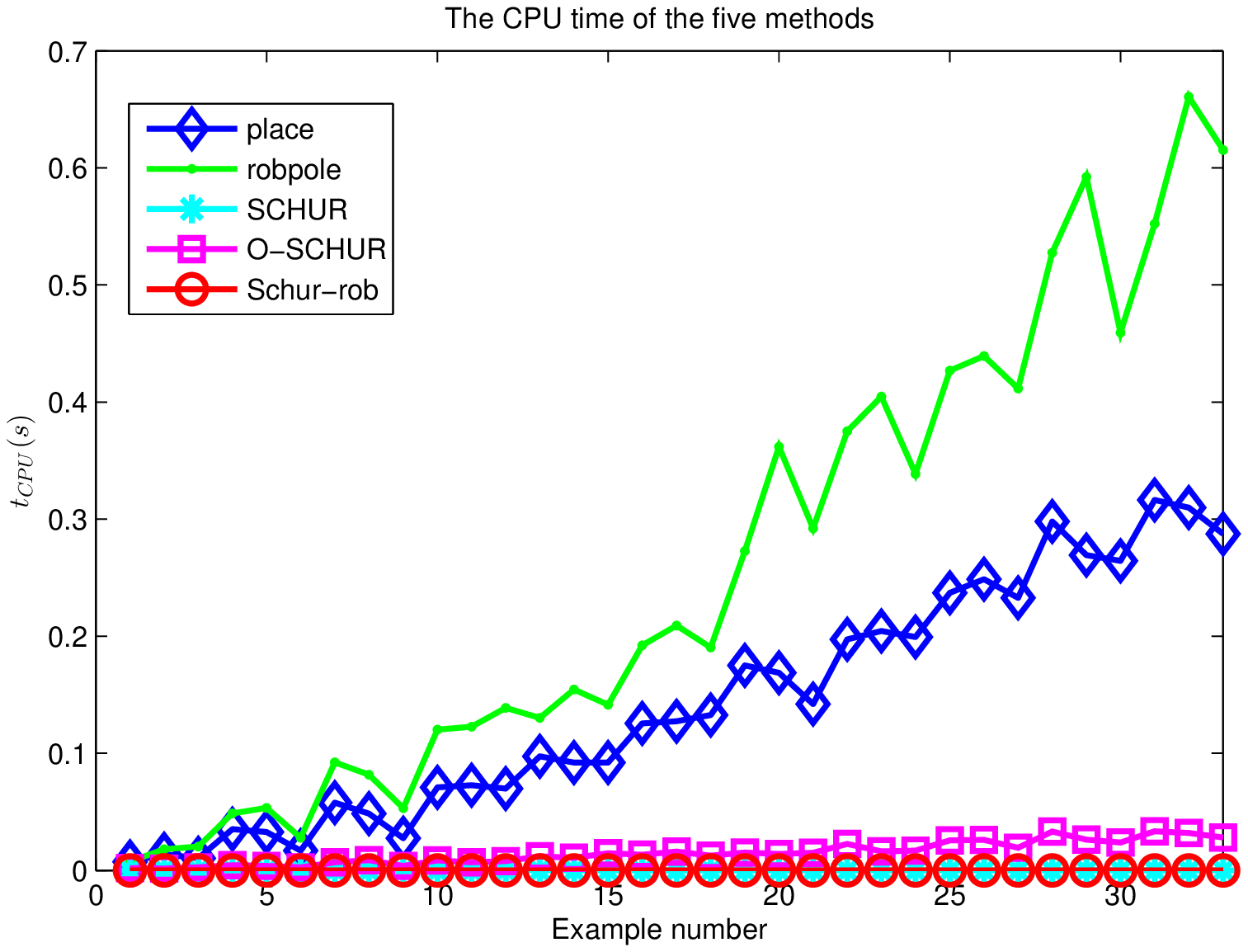}
    \caption{CPU time over 50 trials}\label{fig4}
    \end{minipage}
\end{tabular}
\end{figure}

All these figures show that our \textbf{Schur-rob} algorithm can produce comparable or even better results as
\textbf{place} and \textbf{robpole}, but with much less CPU time. 
}\end{Example}

\section{Conclusion}
Pole assignment problem for multi-input control is generally under-determined. And utilizing this freedom to make the
closed-loop system matrix to be insensitive to perturbations as far as possible evokes the state-feedback robust pole
assignment problem ({\bf SFRPA}) arising. Based on {\bf SCHUR} \cite{Chu2}, we propose a new direct method
to solve the {\bf SFRPA}, which obtains the real Schur form of the closed-loop system matrix and tends to
minimize its departure from normality via solving some  standard eigen-problems. Many numerical examples show that
our algorithm can produce comparable or even better results than existing methods,
but with much less computational costs than the two  classic methods \textbf{place} and \textbf{robpole}.

%\section*{Acknowledgements}
%The authors would like to give thanks to Professor Chu, Professor Tits, Dr. Sima and Dr. Yang for providing
%their MATLAB codes needed in this paper.

\end{document}